\documentclass[10pt,reqno]{amsart}
\usepackage[cp1251]{inputenc}
\usepackage[english]{babel}
\usepackage{amsmath}
\usepackage{amssymb}
\usepackage{amsfonts}
\usepackage{graphicx}
\usepackage{eucal}
\usepackage{hyperref,amsthm}
\usepackage{xcolor}
\RequirePackage{bbold}
\usepackage{mathptmx,bm}
\newtheorem{theorem}{Theorem}
\newtheorem{lemma}{Lemma}

\theoremstyle{definition}

\newtheorem{remark}{Remark}
\numberwithin{equation}{section}
\usepackage{tikz}
\usepackage{pgfplots}

\begin{document}

\title[Asymptotic Analysis of Markov Branching Systems]
    {Enhanced Asymptotic Analysis of Continuous-Time Markov Branching Systems: Revisiting Limiting Structural Theorems}

\author{Azam~A.~Imomov}
\address {Azam~Abdurakhimovich~Imomov
\newline\hphantom{iii} Karshi State University,
\newline\hphantom{iii} 17, Kuchabag st., 180100 Karshi city, Uzbekistan.}
\email{{imomov{\_}\,azam@mail.ru}}

\author{Sarvar~B.~Iskandarov}
\address {Sarvar~B.~Iskandarov
\newline\hphantom{iii} Urgench  State University,
\newline\hphantom{iii} Urgench  city, Uzbekistan.}
\email{sarvar.i@urdu.uz}

\author{Jakhongir~B.~Azimov}
\address {Jakhongir~B.~Azimov
\newline\hphantom{iii} Tashkent State Transport University,
\newline\hphantom{iii} Tashkent 100000, Uzbekistan.}
\email{azimovjb@mail.ru}

\author{Hurshidjon~Q.~Jumaqulov}
\address {Jakhongir~B.~Azimov
\newline\hphantom{iii} Kokand State University, 
\newline\hphantom{iii} Kokand, Uzbekistan.}
\email{xurshid81@gmail.com}

\thanks{\copyright \ 2025 Imomov~A.A. et.al.}




\begin{abstract}
    Markov branching systems form a fundamental class of stochastic models that are extensively applied in biology, physics,
    finance, and other domains. These systems are distinguished by their continuous-time evolution and inherent branching
    structure, allowing transitions to multiple states from a single one. This branching mechanism plays a critical role
    in modeling phenomena such as population dynamics, epidemic spread, and probabilistic systems with multiple outcomes.
    Unlike standard Markov processes, branching systems require a simultaneous treatment of transition dynamics and branching
    probabilities, resulting in a more intricate mathematical framework. In this work, we investigate the asymptotic properties
    of transition functions in continuous-time Markov branching-immigration systems. Our focus lies in refining known limit
    theorems, establishing convergence rates, and deriving improved asymptotic expansions under relaxed moment conditions.
    The results contribute to a deeper understanding of the long-term behavior and invariant structures within these systems.
\end{abstract}

\maketitle

\section{Introduction}            \label{IIBSec1}

    Branching systems serve as important mathematical models for describing population growth, the evolution of networks,
    and other complex systems that exhibit dynamics of reproduction and extinction. Since its origin in the
    early 20th century, these systems have found widespread applications across diverse scientific disciplines and
    practical domains. A crucial subclass of these models consists of Markov branching systems, which are distinguished
    by the property that their evolution depends only on the current state, making them mathematically elegant and
    analytically precise. Markov branching systems are a cornerstone of the theory of stochastic systems, providing
    a robust framework for modeling population dynamics, biological systems, and other phenomena where entities reproduce
    and evolve over continuous time. Rooted in the seminal works {\cite{Pakes75}}, {\cite{Sev57}} and {\cite{Sev51}}, these
    systems extend the classical Galton-Watson branching model to include Markovian dynamics, where the system's future state
    depends solely on its current one. The Markov property provides for a more sophisticated analysis of systems with memoryless
    transitions, making branching systems particularly suitable for applications in genetics, epidemiology, and queuing
    theory. The foundational model describes a closed population where each individual reproduces and dies independently,
    following a specified offspring probability law. Classic results in this area, detailed in monographs {\cite{AsHer}}, {\cite{ANey}},
    {\cite{Bingham}}, {\cite{Harris63}} and {\cite{Jagers75}}, focus on extinction probabilities, growth rates, and long-term
    population behavior. For example, the criticality of the system -- whether the population grows, dies out, or remains
    stable -- is a central theme, and martingale and generating function methods play a key role in the study.

    The inflow of immigrants into the system introduces both complexity and realism to branching models. Immigration
    ensures a continuous influx of new individuals into the population, which can stabilize the system even when
    the internal reproduction process alone would lead to extinction; see {\cite{ChenSPA93}}, {\cite{LiDEDS21}},
    {\cite{LiChPaCh20}}, {\cite{LiChPakes12}}, {\cite{Pakes79}}, {\cite{Pakes76}} and {\cite{Pakes75}}.

    Recent studies, such as those in {\cite{ImSM23}}, {\cite{ImUfa21}}, {\cite{MitovYa25}}, {\cite{MitovYa17}} and
    {\cite{Rahimov21}}, have further advanced our understanding of branching systems allowing immigration, particularly
    in cases where higher-order moments or variances are infinite. These works highlight the asymptotic structure and
    long-term behavior of such systems, providing insights into their stability and scaling limits.

    The literature on Markov branching systems is rich with results that bridge theory and applications. As is well
    known, Sevastyanov~{\cite{Sev57,Sev51}} established fundamental limit theorems for branching systems, laying the
    foundation for subsequent research. The monograph {\cite{Bingham}} made significant contributions to the theory
    of regular variation, which plays a crucial role in analyzing the tail behavior of branching systems. More recently,
    Li et al.~{\cite{LiDEDS21,LiChPaCh20}} investigated large deviation rates and the long-term behavior of Markov
    branching-immigration systems, providing new perspectives on the interplay between reproduction and immigration
    mechanisms. Furthermore, Pakes~{\cite{Pakes99}} revisited conditional limit theorems, offering deeper insights
    into the probabilistic structure of these systems.

    Thus, Markov branching models serve as powerful analytical tools for studying stochastic systems with complex dynamics.
    The extensive body of research, ranging from seminal contributions to modern advancements, underscores the lasting
    significance of these systems in both theoretical and applied contexts.

    This report contributes to the theory of homogeneous continuous-time Markov branching systems by refining classical
    results and establishing new asymptotic properties under relaxed moment assumptions. We focus in particular on the
    limiting structure of transition functions, the convergence to invariant measures, and the explicit characterization
    of the convergence rates. The analysis is conducted within a rigorous framework based on the concept of slowly
    varying functions in the sense of Karamata.

    The rest of the paper is structured as follows. In {\hyperref[IIBSec2]{Section}~\ref{IIBSec2}}, we describe the
    structure of the Markov branching system and briefly discuss the classification of states. Our investigation
    is motivated by the desire to deepen the theoretical understanding of such systems and to improve the precision
    of existing asymptotic results. Accordingly, we define the principal goals of the study and, to pursue them,
    we adopt a set of Basic axioms, adhering strictly to the fundamental principles underlying infinitesimal moment
    conditions. {\hyperref[IIBSec3]{Section}~\ref{IIBSec3}} presents the asymptotic expansion of the generating
    function and its derivative for the Markov branching systems. These representations underlie our main results,
    which will be stated and proved in {\hyperref[IIBSec4]{Section}~\ref{IIBSec4}}.

\section{\textbf{Preliminaries}}             \label{IIBSec2}

\subsection{\textbf{Description of the model}}             \label{IIBSubsec2.1}

    We investigate the population growth dynamics of monotype individuals governed by the Markov branching-immigration
    system (MBIS) evolving over a continuous-time axis $\mathcal{T}:=(0,+\infty)$. The state of the system is determined
    by the total population size at time $t$. The system is modeled as a stochastic process incorporating both reproduction
    and immigration mechanisms. Each individual lives for a random amount of time and produces offspring independently
    according to a fixed probabilistic law, while new individuals may enter the population at random times due to
    an external immigration process. This setting captures more realistic scenarios encountered in biological
    systems, such as microbial populations or epidemic models, where population size is influenced by both
    internal reproduction and external influx.

    The population size evolves according to the following scheme. Let ${\mathbb{N}}$ be a set of natural numbers
    and ${\mathbb{N}}_0=\{0\}\cup{\mathbb{N}}$. Each individual existing at time $t\in{\mathcal{T}}$ in the population,
    independently of their history and of each other, behaves independently and undergoes transitions according to
    the following rules during an infinitesimal time interval $(t, t+\varepsilon)$ as $\varepsilon\downarrow{0}$:
\begin{itemize}
\item with probability $a_j\varepsilon+o(\varepsilon)$ it produces $j\in{\mathbb{N}}_0\backslash\{1\}$ offspring,
        where $a_j>0$ are the branching rates and the case $j=1$ is excluded;

\item alternatively, with probability $1+a_1\varepsilon+o(\varepsilon)$, it either survives without reproducing
        or produces exactly one offspring, each event occurring with equal chance.
\end{itemize}
    Imposing the normalization condition $\lambda\sum_{j}{a_j}=1$, for some $\lambda>0$
    and for $j\in{\mathbb{N}}_0\backslash \{1\}$, we ensure that the individual lifetime is exponentially distributed
    with mean $\lambda$. Under this assumption, the branching rates satisfy the consistency condition
\[
    0<\lambda{a_0}<-\lambda{a_1}=1  \qquad {\text{and}} \qquad a_j\ge 0 \quad
    {\text{for}} \quad {j\in{\mathbb{N}}_0 \backslash \{1\}}.
\]
    Hence, upon death, an individual produces $k\in{\mathbb{N}}_0 \backslash \{1\}$ offspring
    with probability $\lambda{a_k}$, independently of the rest of the population.

    In addition to reproduction, the system incorporates immigration. During the same infinitesimal
    time interval, an independent inflow of individuals may occur:
\begin{itemize}
\item with probability $b_k\varepsilon+o(\varepsilon)$ a group of $k\in{\mathbb{N}}_0\backslash\{1\}$
        new individuals immigrates independently of the branching events into the population;

\item with the remaining probability $1-\sum_{k\in{\mathbb{N}}}{b_k\varepsilon}+o(\varepsilon)$,
        no immigration occurs.
\end{itemize}
    The immigration intensities $b_k\geq{0}$ for $k\in{\mathbb{N}}$ and $0<-b_0=\sum_{k\in{\mathbb{N}}}{b_k}<\infty$,
    ensuring that the total immigration rate is finite. Each immigrant is assumed to behave identically
    to aboriginal individuals after arrival, following the same branching dynamics and lifetime distribution.

    This formulation corresponds to the infinitesimal generator of a continuous-time Markov branching system,
    capturing the transition dynamics over vanishingly small time intervals. The coefficients $a_j$ and $b_j$
    encode, respectively, the branching rates and immigration intensities, while the terms $o(\varepsilon)$
    indicate higher-order infinitesimal probabilities negligible as $\varepsilon\downarrow{0}$. This local
    probabilistic structure effectively models both the variability in reproduction and the possibility of
    survival without reproduction, thereby encapsulating the fundamental stochastic features of the population size evolution.

    Thus, the overall population growth mechanism is fully characterized by two infinitesimal generating functions (GF):
\begin{itemize}
\item the offspring GF $f(s):=\sum_{j}{a_j s^j}$ governs the reproduction behavior of individuals,
        encoding the infinitesimal probabilities of producing offspring; and

\item the immigration GF $h(s):=\sum_{k}{b_k s^k}$ describes the probabilistic structure of the immigration
        flow into the system, reflecting the rates at which external individuals enter the population.
\end{itemize}
    They succinctly encode the local stochastic dynamics and serve as a foundation for further analytical
    treatment, including the study of the GF of the system, its moment dynamics, and asymptotic behavior.

    Denote $X(t)$ the population size at time $t\in\mathcal{T}$ in MBIS. The system is modeled as an irreducible,
    homogeneous continuous-time Markov chain with state space $\mathcal{E}\subset{{\mathbb{N}}_0}$ and transition functions
\[
    {p_{ij}(t)}:={\mathbb{P}}\bigl\{{X(\tau + t)=j\bigm\vert{X(\tau)=i}}\bigr\} \qquad {\text{for any}} \quad \tau,t\in\mathcal{T},
\]
    and for all $i,j\in\mathcal{E}$. The irreducibility of the chain implies that any state can be reached from any other state
    in the system, which ensures the overall connectivity of the process. The transition probabilities ${p_{ij}(t)}$ encapsulate
    the stochastic dynamics of the system, describing the probability of transition from state $i\in\mathcal{E}$ to state
    $j\in\mathcal{E}$ in an infinitesimal time interval, thereby governing the flow of population change. Moreover, the system
    is characterized by branching rates $\left\{a_j\right\}$, where each rate corresponds to the probabilistic dynamics of population
    growth or decline in each state. This branching structure introduces an element of immigration dynamics into the system
    growth, allowing for the evolution of population size according to the probabilistic laws inherent in branching models.

    Recall that the GF ${\mathcal{P}_{i}(t;s)}:=\sum_{j\in{\mathcal{E}}}p_{ij} (t)s^{j}$ admits the representation
\begin{equation}                            \label{IIB2.1}
    {\mathcal{P}_{i}(t;s)}=\left[F(t;s)\right]^{i}\exp\left\{\int_{0}^{t}h\left(F(u;s)\right)du\right\},
\end{equation}
    where $F(t;s)$ satisfies the backward
\begin{equation}                           \label{IIB2.2}
    \frac{\partial{F(t;s)}}{\partial{t}}=f\left(F(t;s)\right),
\end{equation}
    and the forward
\begin{equation}                           \label{IIB2.3}
    \frac{\partial{F(t;s)}}{\partial{t}}=f(s)\frac{\partial{F(t;s)}}{\partial{s}}
\end{equation}
    Kolmogorov equations with the initial condition $F(0;s)=s$; for details, see {\cite{Harris63}}, {\cite{Sev71}}.

    The case $h(0)=1$ corresponds to the absence of any immigration flow in the system, indicating that the dynamics are
    entirely governed by the evolution of the native population without external influx. In this case, the population
    growth evolves within a closed system, and the influence of potential immigrant contributions is eliminated from the
    model. In this setting, our system simplifies to a more structured yet mathematically rigorous model, referred to as
    the Markov Branching System (MBS) without immigration, whose evolution is governed solely by the branching rates
    $\left\{a_{j},j\in{\mathbb{N}}_{0}\right\}$. This formulation corresponds to the classical continuous-time branching
    system, originally introduced to describe the reproduction dynamics of homogeneous populations, and serves as a
    foundational model in the probabilistic theory of population growth; see {\cite{ANey}}, {\cite{Harris63}}, {\cite{Sev71}}.
    Let $Z(t)$ denote the population size at time $t\in{\mathcal{T}}$ in MBS.
\begin{remark}                  \label{IIBRem1}
    The MBS $\left\{Z(t)\right\}$ is a reducible, continuous-time Markov chain with the state space
    ${\mathcal{S}}_{0}=\{0\}\cup{\mathcal{S}}$, where $\{0\}$ is an absorbing state and ${\mathcal{S}}\subset{\mathbb{N}}$
    forms a communicating class of transient states. According to the classical theory of Markov chains, the system is
    classified as reducible if it is possible to transition from a transient class to an absorbing state. In this case,
    once the process enters the absorbing state $\{0\}$, it remains there indefinitely, signifying the system's extinction.
    Although the class ${\mathcal{S}}$ forms a communicating class of transient states, the inevitable transition to the
    absorbing state $\{0\} $ aligns with the fundamental principles of Markov theory. Consequently, the model is consistent
    with Markov chain theory, where ${\mathcal{S}}$ is the transient class, and $\{0\}$ serves as the absorbing state.
    The system, being reducible, eventually reaches absorption in $\{0\}$,  marking an extinction.
\end{remark}

    By equations {(\ref{IIB2.2})} and {(\ref{IIB2.3})}, the function $f(s)$ naturally induces a family of
    time-dependent GF $F(t,s)$, which describes the temporal evolution of the system. In particular, the
    transition function $P_{ij}(t):={\mathbb{P}_i}\left\{Z(t)=j\right\}$ from state $i$ to state $j$ in
    time $t$ is given by $\textsf{p}_{j}(t):=P_{1j}(t)$, i.e.
\[
    {{\mathbb{E}}_{i}}\left[{s}^{Z(t)}\right]:=
    \sum_{j\in{{\mathcal{S}}_{0}}}{{P}_{ij}(t){s}^{j}}={{\bigl[F(t;s)\bigr]}^{i}},
\]
    where $F(t;s)=\sum_{j\in{\mathcal{S}}_{0}}{\textsf{p}_{j}(t){s}^{j}}$.

    Wherever necessary, we will write ${\mathbb{P}}$ and ${\mathbb{E}}$ instead of ${\mathbb{P}_1}$ and ${\mathbb{E}}_1$.

    We observe that $q:={{\lim}_{t\to\infty}}F(t;0)=\mathbb{P}\left\{Z(t)=0\ \, for\ some \ \, t\in\mathcal{T}\right\}$
    is the probability that a genetic lineage, originating  by the single individual, eventually dies out. Furthermore,
    this probability is the smallest root of the equation $f(s)=0$ for $s\in[0,1]$. It is known $F(t;s)\to{q}$
    as $t\to\infty$ uniformly  in $s\in[0,d],\,\,d<1$; see {\cite[Ch.I\!I\!I]{ANey}}. Then the formula {(\ref{IIB2.1})}
    implies that ${{{\mathcal{P}}_{i}}(t;s)}\big/{\mathcal{P}(t;s)}\to{q}$. Therefore, we can restrict our study to
    $\mathcal{P}(t;s):={{\mathcal{P}}_{0}}(t;s)$ if $t\to\infty$. The population mean can then be calculated by
    differentiating equation {(\ref{IIB2.1})}, which gives:
\begin{eqnarray}                  \label{IIB2.4}
    \mathbb{E}X(t)={{\left.\frac{\partial \mathcal{P}(t;s)}{\partial{s}}\right|}_{s=1}}=
\left\{
\begin{array}{l}
    \displaystyle\frac{{h}'({{1}^{-}})}{a}\left( {{e}^{at}}-1 \right)
    \qquad \hfill {\text{if}} \quad {a\ne{0},}
\vspace{2.8mm}  \\
    {h}'({{1}^{-}})t   \qquad \hfill {\text{if}} \quad {a=0,}\\
\end{array}
\right.
\end{eqnarray}
    where the prime denotes the derivative, and $a={f}'({{1}^{-}})$; from here on we use the notation
    $A({{1}^{-}}):={{\lim }_{s\uparrow 1}}A(s)$. The formula {(\ref{IIB2.4})} implies that the asymptotic behavior
    of $\mathbb{E}X(t)$ is governed by the sign of the parameter $a$, resulting in qualitatively different trajectories.
    In this context, the chain $\left\{X(t)\right\}$ is categorized as subcritical, critical, or supercritical
    when $a<0$, $a=0$ and $a>0$, respectively. Accordingly, the classification of MBS follows the same principles.
    In the case of $a\le{0}$, all generation lines in the MBS eventually die out. More precisely, $q=1$
    when $a\le{0}$, whereas $q<1$ when $a>0$.

\subsection{\textbf{Background and Motivation}}             \label{IIBSubsec2.2}

    Given the above, a comprehensive study of the trajectory of $\left\{X(t)\right\}$ requires an exact characterization
    of the asymptotic behavior of the transition function ${{p}_{ij}}(t)$, as well as a rigorous analysis of the ergodicity
    property of the state space and the existence of an invariant measure for MBIS. In critical case,
    Sevastyanov~{\cite{Sev57}} established that if $2b:=f''({{1}^{-}})<\infty$ and $h'({{1}^{-}})<\infty$, then the
    normalized system ${X(t)}/{bt}$ converges to a limiting standard Gamma distribution function. In the case of MBS,
    this reduces to the standard exponential distribution. Pakes~{\cite{Pakes75}} established that
    ${t^{\lambda}}{{\mathcal{P}}_i}(t;s)$ converges to a limit function ${\mathcal{U}}_{\lambda}(s)$ as $t\to\infty$.
    This function admits a power series expansion ${\mathcal{U}}_{\lambda}(s)=\sum_{j\in\mathcal{E}}{u_j{s}^j}$
    and generates an invariant measure providing that
\begin{equation}                              \label{IIB2.5}
    \sum\limits_{j\in \mathbb{N}}{{{a}_{j}}{{j}^{2}}\ln j}<\infty \qquad {\text{and}} \qquad
    \sum\limits_{j\in \mathbb{N}}{{{b}_{j}}j\ln j}<\infty.
\end{equation}
    Ergodic properties of the states of MBIS obey the corresponding properties of an arbitrary continuous-time Markov chain,
    which Anderson's monograph {\cite[Ch.6]{Anderson}} studies in detail. Sevastyanov obtained the first fundamental results
    on the existence of invariant measures for MBIS in his pioneering research {\cite{Sev57}}. Conner~{\cite{Conner67}}
    examined the invariant properties of the critical MBIS, and later Seneta~{\cite{Seneta74}} established a unique
    correspondence between the properties of the invariant measure of branching systems with and without immigration
    in the discrete-time case. Yang~{\cite{Yang72}}focused on the subcritical case, followed him, Pakes~{\cite{Pakes75}}
    provided a comprehensive classification of all possible cases for the trajectories of MBIS.
    In {\cite{LiChPakes12}} later extended and refined results of {\cite{Pakes75}}.

    According to a corresponding result in {\cite{LiChPakes12}}, the invariant measure of the MBIS can also be constructed via the strong ratio limit property of transition functions. Specifically, the sequence of positive numbers
\[
    \left\{{\pi_{j}}:={\lim_{t\to\infty}}{{p_{0j}}(t)}\big/{{p_{00}}(t)}\right\}
\]
    for all $j\in\mathcal{E}$
    forms an invariant measure. Note the close connection between the sequences $\left\{{u_j},j\in\mathcal{E}\right\}$
    and $\left\{{\pi_{j}},j\in\mathcal{E}\right\}$, as well as the corresponding connection between their GFs.
    Fundamentally, these representations correspond to different formulations of the same limiting law. It follows
    directly that ${{\mathcal{U}}_{\lambda}}(s)={{\mathcal{U}}_{\lambda}}(0)\pi(s)$, where
    $\pi(s)=\sum_{j\in\mathcal{E}}{{\pi_{j}}{s^j}}$. This is in full agreement with the
    uniqueness (up to a multiplicative constant) of the invariant measure of the MBIS.

    We focus exclusively on the critical case. In the theory of critical MBIS, among the various problems related to
    the asymptotic properties of the system trajectories, one issue stands out as being of exceptional importance.
    It involves the study of the asymptotic approach of transition functions to invariant measures and the explicit
    construction of the corresponding GFs. This task was first raised in the paper {\cite{Pakes75}}, in which the
    author proved that under condition {\eqref{IIB2.5}} the function $t^{\lambda}{\mathcal{P}}_i(t;s)$ approaches
    a limit function ${{\mathcal{U}}_{\lambda}}(s)$ as $t\to\infty$, which has the form
\begin{equation}                              \label{IIB2.6}
    {{\mathcal{U}}_{\lambda}}(s)=\frac{1}{{[b(1-s)]}^{\lambda}}
    \exp\left\{\int_{s}^{1}{\left[\frac{h(u)}{f(u)}-\frac{\lambda}{1-u}\right]du}\right\}.
\end{equation}
    Significant refinement of the above results was achieved in {\cite{ImJSFU14}}, where estimates for the rate
    of convergence of ${t^{\lambda}}{{\mathcal{P}}_{i}}(t;s)$ to ${{\mathcal{U}}_{\lambda}}(s)$ were established.
    The results of these studies were obtained under the conditions of the finiteness of higher-order
    moments $f'''(1^{-})$ and $h''(1^{-})$.

    We are now motivated to enhance the above-recalled results from {\cite{ImJSFU14}} and {\cite{Pakes75}} by relaxing
    the finiteness of the higher-order moments assumptions of the branching and immigration rates. In doing so, we
    intend to derive a more refined and improved asymptotic expansion of the function $F(t;s)$. For this purpose,
    we will make substantial use of approaches based on the slow (or, more generally, regular) varying conception
    in the sense of Karamata; see, {\cite{AsHer}}, {\cite{Bingham}} and {\cite{SenetaRV}}. Recall that a function
    $L(\cdot)$ is called slowly varying at a point $\alpha\in\mathbb{R}$ if it is positive and the ratio
    ${L(\lambda{t})}/{L(t)}$ approaches $1$ as $t\to\alpha$ for each $\lambda>0$. In what follows, we will denote
    the class of slowly varying functions at a point $\alpha$ by ${\mathcal{SV}}_{\alpha}$. The representation
    theorem states that every function $L(\cdot)\in{\mathcal{SV}}_{\alpha}$ can be expressed
    as $L(\cdot)=c(x)\exp\int_{a}^{x}{\left[{\varepsilon(u)}/{u}\right]du}$ for some $a>0$, where $c(x)$ is a bounded
    function such that $c(x)\to{c}>0$ and $\varepsilon(x)$ is a continuous function satisfying $\varepsilon(x)\to{0}$
    as $x\to\infty$. If $c(x)\equiv{c}$, then $L(\cdot)$ is called a normalized function. A function $V(\cdot)$ is
    said to be regularly varying at infinity with index $\rho$ if it is defined on $(0,\infty)$ and can be
    represented as $V(x)={x^{\rho}}L(x)$ for some $L(\cdot)\in{\mathcal{SV}}_{\infty}$.

\subsection{\textbf{Basic Axioms and aim}}             \label{IIBSubsec2.3}

    Next, we impose the following conditions on the GFs $f(s)$ and $h(s)$. Let
\[
    f(s)={{(1-s)}^{1+\nu }}\mathcal{L}\left( \frac{1}{1-s} \right)  \eqno[\textsf {$f_{\nu}$}]
\]
    and
\[
    h(s)=-{{(1-s)}^{\delta }}\ell \left( \frac{1}{1-s} \right)    \eqno[\textsf {$h_{\delta}$}]
\]
    for all $s\in [0,1)$, where $\nu,\delta\in(0,1)$ and, the functions $\mathcal{L}(\cdot)$, $\ell(\cdot)$ are in
    the Karamata functions class ${{\mathcal{SV}}_{\infty}}$. This assumptions state that the offspring distribution
    belongs to the domain of attraction of the $(1+\nu)$-stable law, whereas the immigration distribution belongs
    to the domain of attraction of the $\delta$-stable law. By the criticality of the system, the assumption
    ${{f}_{\nu}}$ implies that $2b:=f''(1^{-})=\infty$. If $b<\infty$ then $[f_{\nu}]$ holds for $\nu=1$ and
    $\mathcal{L}(t)\to{b}$ as $t\to\infty$. Similarly, the GF $h(s)$ of the form $[h_{\delta}]$ generates the
    immigration law, which has the $\delta$-order moment. However, if $h'(1^{-})<\infty$, then $[h_{\delta}]$ is
    satisfied with $\delta=1$ and $\ell(t)\to{h}'(1^{-}$ as $t\to\infty$. It was established in {\cite{LiChPakes12}}
    that $\mathcal{S}$ is positive-recurrent if $\gamma:=\delta-\nu>0$, and transient if $\gamma<0$. In the special
    case where $\gamma=0$, it follows that $h(s)={f}'(s)$ and $\mathcal{L}(t)\to{1+\nu}$ as $t\to\infty$.
    Under this condition, a different population system, known as the Markov Q-process, replaces the MBIS.
    For details on the Markov Q-process, we refer the reader to {\cite{ImPMS19}}, {\cite{Im2012}}, {\cite{Pakes99}}.

    Further, since $\mathcal{L}(\lambda{t})\sim\mathcal{L}(t)$, the approaching rate of the ratio
    ${\mathcal{L}(\lambda{t})}/{\mathcal{L}(t)}$ to unity should be characterized by some infinitesimal
    function as $t\to\infty$, uniformly in $\lambda>0$. Thus, we can formally write that
\[
    \frac{\mathcal{L}(\lambda{t})}{\mathcal{L}(t)}=1+{{\omega}_{\lambda}}(t),  \eqno[\textsf {${{\mathcal{L}}_{\nu}}$}]
\]
    where ${{\omega}_{\lambda}}(t)$ is an infinitesimal function, known as the remainder term, which vanishes
    at infinity uniformly with respect to $\lambda$. This allows us to characterize the behavior of slowly
    varying functions in a more precise manner. In the presence of the remainder term $\omega(t):={{\omega}_{\lambda}}(t)$,
    it is said that the function $\mathcal{L}(\cdot)\in{{\mathcal{SV}}_{\infty}}$ belongs to the class of slowly
    varying functions with the remainder $\omega(t)$; see {\cite[p.185]{Bingham}}. For this specific class
    of functions, we adopt the convention $\mathcal{L}(\cdot)\in{{\mathcal{SV}}_{\infty}}(\omega)$, in
    line with the standard notation commonly used in the field. Whenever condition $[{\mathcal{L}}_{\nu}]$ is
    invoked, we will consistently assume that
\[
    \omega (t)=\mathcal{O}\left(\frac{\mathcal{L}(t)}{t^{\nu}}\right)
    \qquad {\text{as}} \quad {t\to\infty}.   \eqno[\textsf {${{\omega}_{\nu}}$}]
\]
    Similarly, we impose an extra condition on the function $\ell(\cdot)$, assuming that
\[
    \frac{\ell(\lambda{t})}{\ell(t)}=1+{{\psi}_{\lambda}}(t),   \eqno[\textsf {${{\ell}_{\delta}}$}]
\]
    so that it belongs to the class ${{\mathcal{SV}}_{\infty}}(\psi)$ with remainder
\[
    \psi(t):={{\psi}_{\lambda}}(t)=\mathcal{O}\left( \frac{\ell(t)}{t^{\delta}}\right)
    \qquad {\text{as}} \quad {t\to\infty}   \eqno[\textsf {${{\psi}_{\delta}}$}]
\]
    for each $\lambda>0$. In what follows, we consider assumptions $[{{f}_{\nu}}]$, $[{{h}_{\delta}}]$,
    $[{{\mathcal{L}}_{\nu}}]$, $[{{\omega}_{\nu}}]$, $[{{\ell}_{\delta}}]$ and $[{{\psi}_{\delta}}]$
    as our \textit{Basic Axioms}.

    In the papers {\cite{ImMDPI24}}, {\cite{ImSM23}}, {\cite{ImUfa21}}, {\cite{ImPMS19}} we have demonstrated
    that the Basic Axioms play a fundamental role in elucidating the intricate asymptotic properties of
    Karamata functions within the class ${{\mathcal{SV}}_{\infty}}(\omega)$. These axioms establish
    a rigorous foundation for analyzing the fine structure of slowly varying functions at infinity,
    providing deeper insights into their limiting behavior. In our case, the Basic Axioms serve as
    a powerful analytical framework for investigating the asymptotic properties of branching systems,
    enabling a precise characterization of their long-term dynamics and scaling laws.

    As for MBIS, it was demonstrated in the paper {\cite{ImUfa21}} that, under the Basic Axioms, the asymptotic
    behaviors of the transition functions ${{p}_{ij}}(t)$ are regulated by the parameter $\gamma$, depending
    on its sign. In the special case when the state space $\mathcal{S}$ is transient, $\gamma<0$, it was shown
    that there exists a function $A(t)$ such that the limit function
    $\mathcal{U}(s):=\lim_{t\to\infty}A(t)\mathcal{P}(t;s)$ has an explicit form, analogous to
    expression {\eqref{IIB2.6}}, which depends on GFs $f(s)$ and $h(s)$. Moreover, its power series
    expansion of $\mathcal{U}(s)$ generates an invariant measure with respect to the transition
    functions ${{p}_{ij}}(t)$.

    The aim of this report is as follows. For simple MBS, we establish significant
    refinements of some limit structure theorems. For transient MBIS, we rigorously determine the exact rate
    of convergence of the function $A(t)\mathcal{P}(t;s)$ to the limiting function $\mathcal{U}(s)$, explicitly
    identifying the normalizing term $A(t)$ as a function of the parameter $\gamma$. The asymptotic expansion
    of the function $F(t;s)$ plays a principal role in achieving our aim. In the theory of critical MBS,
    we shall refer to this statement as the Basic Lemma. {\hyperref[IIBSec3]{Section}~\ref{IIBSec3}} will
    focus on its rigorous development, providing a comprehensive and in-depth analysis of its construction.

\section{\textbf{Basic Lemma of critical MBS theory and its consequences}}             \label{IIBSec3}

    Consider critical MBS $\left\{Z(t),t\in\mathcal{T}\right\}$ generated by GF $f(s)$. In our designations
\[
    F(t;s)=\mathbb{E}{{s}^{Z(t)}}=\sum_{j\in{\mathcal{S}}_{0}}{\textsf{p}_{j}(t){s}^{j}},
\]
    where ${\textsf{p}_{j}(t)}=\mathbb{P}\left\{Z(t)=k\right\}$. Our primary focus in this section will be
    on the statement regarding the asymptotic expansion of the function
\[
    \tau(t;s):=\frac{1}{R(t;s)} {\raise1.5pt\hbox{,}}
\]
    where $R(t;s)=1-F(t;s)$. In classical theory, such a statement is traditionally called the Basic
    Lemma of the theory of critical MBS. In this context, the results established in papers {\cite{ImUfa21}}
    and {\cite{ImMJMS17}} under condition $[{f_{\nu}}]$, imply that the function $\tau(t;s)$ admits
    for all $s\in[0,1)$ the following expansion:
\begin{equation}                              \label{IIB3.1}
    \tau(t;s)=\frac{{(\nu t)}^{1/{\nu}}}{\mathcal{N}(t)}{{\left[1+\frac{\mathcal{M}(s)}{t}\right]}^{1/{\nu}}},
\end{equation}
    where the function $\mathcal{N}(\cdot)\in\mathcal{S}{{\mathcal{V}}_{\infty}}$ is defined by the property
\begin{equation}                              \label{IIB3.2}
    {{\mathcal{N}}^{\nu}}(t)\mathcal{L}\left(\frac{{(\nu{t})}^{1/{\nu}}}{\mathcal{N}(t)}\right)\to{1}
    \qquad {\text{as}} \quad {t\to\infty},
\end{equation}
    and $\mathcal{M}(s)$ is GF of invariant measures of MBS and
\[
    \mathcal{M}(s)=\int_{0}^{s}{\frac{dx}{f(x)}}=
    \int_{1}^{{1}/{(1-s)}}{\frac{dx}{{x^{1-\nu}}\mathcal{L}(x)}} {\raise1.5pt\hbox{.}}
\]
    The primary aim of this section is to improve statement {(\ref{IIB3.1})} based on the Basic Axioms.

    Given that the remainder term $\omega(x)$ belonging to the class $[{{\omega}_{\nu}}]$ decays at the rate
    $\mathcal{O}\left({\mathcal{L}(x)}/{{x^{\nu}}}\right)$, we are fully justified in refining our analysis
    by further narrowing our focus to the more specific asymptotic regime
\[
    \frac{\omega (x)}{\mathcal{L}(x)}{{x}^{\nu }}\to{1}   \qquad {\text{as}} \quad {x\to\infty}.
\]
    Since $\mathcal{L}(\cdot)\in{{\mathcal{SV}}_{\infty}}(\omega)$, the results in {\cite{ImMDPI24}} ensure that ${{\lim}_{x\to\infty}}\mathcal{L}(x)<\infty$. Consequently, the ``narrowed'' asymptotic condition above entails that
\begin{equation}                              \label{IIB3.3}
    {{\lambda}_{\omega}}:=\lim_{x\to\infty}{x^{\nu}}\omega(x)<\infty.
\end{equation}

    We note that a similar narrowing of asymptotic conditions is expected to emerge for the function
    $\ell(\cdot)\in{{\mathcal{SV}}_{\infty}}(\psi)$ in the course of our forthcoming analysis.

    In what follows, it will become clear that such a narrowing of the asymptotic conditions is expected to play a decisive
    role in obtaining ``outstanding'' results, while strictly observing the fundamental principles of the Basic Axioms and
    ensuring the preservation of the asymptotic decay rates of the remainder terms $\omega(\cdot)$ and $\psi(\cdot)$.

    Let
\[
    \Lambda(y):={y^{\nu}}\mathcal{L}\left(\frac{1}{y}\right) \qquad {\text{for}} \quad {y\in(0,1]}.
\]

\begin{lemma}                  \label{IIBLem1}
    Let conditions $[{f_{\nu}}]$, $[{{\mathcal{L}}_{\nu}}]$ and $[{{\omega}_{\nu}}]$ be satisfied. Then
\begin{equation}                              \label{IIB3.4}
    \frac{1}{\Lambda(R(t;s))}-\frac{1}{\Lambda(1-s)}=\nu{t}+
    \int_{\tau(0;s)}^{\tau(t)}{\frac{\varepsilon(x)}{{x^{1-\nu}}\mathcal{L}(x)}dx}+
    \mathcal{O}\left(1/t\right)
\end{equation}
    as $t\to\infty$, where $\tau(t)={{(\nu{t})}^{1/\nu}}\big/{\mathcal{N}(t)}$ and $\mathcal{N}(t)$ is defined
    in {(\ref{IIB3.2})}. By incorporating the supplementary constraint {(\ref{IIB3.3})}, we directly derive the
    following fundamental identity:
\begin{equation}                              \label{IIB3.5}
    \frac{1}{\Lambda(R(t;s))}-\frac{1}{\Lambda(1-s)}=\nu{t}-\frac{1}{\nu}\ln\left[\Lambda(1-s)\nu{t}\right]+\rho(t;s),
\end{equation}
    where $\rho(t;s)=o(\ln{t})$ as $t\to\infty$ uniformly in $s\in[0,1)$.
\end{lemma}

\begin{proof}
    First, we recall that the function $y\Lambda(y)$ is positive and approaches zero. And also
    it has a monotone derivative such that ${y{\Lambda}'(y)}\big/{\Lambda(y)}\to\nu$ as $y\downarrow{0}$;
    see [Bing, p. 401]. For this reason we will be firmly convinced that
\begin{equation}                              \label{IIB3.6}
\sigma(y):=\nu -\frac{y{\Lambda }'(y)}{\Lambda (y)}
\end{equation}
    is continuous and infinitely small function as $y\downarrow{0}$. We see that $\Lambda(1)={{p}_{0}}$,
    then integrating {(\ref{IIB3.6})} over $(1,y]$ it follows
\[
    \Lambda(y)={{a}_{0}}{{y}^{\nu}}\exp\left\{-\int_{1}^{y}{\frac{\sigma(u)}{u}du}\right\}.
\]
    Dividing by ${{y}^{\nu}}$, the last equality becomes
\[
    \mathcal{L}\left({1}\big/{y}\right)={{a}_{0}}\exp\left\{-\int_{1}^{y}{\frac{\sigma(u)}{u}du}\right\};
\]
    alternatively, by substituting $u={1}/{\tau}$ into the integrand, it takes the form
\begin{equation}                              \label{IIB3.7}
    \mathcal{L}(x)={{a}_{0}}\exp\int_{1}^{x}{\frac{\varepsilon(\tau)}{\tau}d\tau},
\end{equation}
    where $\varepsilon(\tau)=\sigma\left({1}/{\tau}\right)$.
    Combining $[{{\mathcal{L}}_{\nu}}]$ and {(\ref{IIB3.7})} entails
\begin{equation}                              \label{IIB3.8}
    \int_{x}^{\lambda x}{\frac{\varepsilon (\tau )}{\tau }d\tau }=\ln \left[ 1+\omega (x) \right].
\end{equation}
    Using an expansion of $\ln[1+y]$ as $y\downarrow{0}$ in right-hand side of {(\ref{IIB3.7})} and
    applying the mean value theorem to the left-hand side of the last equality produces the following relation:
\begin{equation}                              \label{IIB3.9}
    \varepsilon({{\zeta}_{x}})\ln\lambda=\omega(x)-\frac{1}{2}{{\omega}^{2}}(x)\left(1+o(1)\right)
    \qquad {\text{as}} \quad {x\to\infty},
\end{equation}
    where
\begin{align*}
\left\{
\begin{array}{l}
    x<{{\zeta}_{x}}<\lambda{x}   \qquad \hfill {\text{if}} \quad {\lambda>1,}
\vspace{2.8mm}  \\
    \lambda{x}<{{\zeta}_{x}}<x   \qquad \hfill {\text{if}} \quad {\lambda<1.}\\
\end{array}
\right.
\end{align*}
    By the mean value theorem, we can write that ${{\zeta}_{x}}=\left(1-(1-\lambda)\theta\right)x$ for some $\theta\in(0,1)$. Then
    $\mathcal{L}({\zeta}_{x})=\mathcal{L}(x)\left(1+\omega(x)\right)$ and hence we obtain that
\[
    \varepsilon({{\zeta}_{x}})=\omega({\zeta}_{x})\left(1+\mathcal{O}\left(\frac{\mathcal{L}(x)}{{x}^{\nu}}\right)\right)
    \qquad {\text{as}} \quad {x\to\infty}.
\]
    Therefore, equation {(\ref{IIB3.9})} becomes
\[
    \varepsilon(x)=\omega(x)\left(1+\mathcal{O}(\omega(x))\right)  \qquad {\text{as}} \quad {x\to\infty}.
\]
    Thus, according to our notation, we have
\begin{equation}                              \label{IIB3.10}
    \sigma(y)=\omega\left({1}\big/{y}\right)\left(1+\mathcal{O}\left(\Lambda(y)\right)\right)
    \qquad {\text{as}} \quad {y\downarrow{0}}.
\end{equation}

    Returning to the function $\sigma(y)$, defined in {(\ref{IIB3.6})}, and replacing $y$ with $R:=R(t;s)$ there,
    we rewrite it as follows:
\[
    \nu-\frac{R{\Lambda}'(R)}{\Lambda(R)}=\sigma(R).
\]
    Combining this equality with condition $[{{f}_{\nu }}]$ and the backward Kolmogorov equation
    ${\partial{F}}\big/{\partial{t}}=f(F)$, while accounting for the definition of $\Lambda (y)$,
    we derive the following relations:
\begin{eqnarray*}
    {\frac{d\Lambda(R)}{dt}}
& = & \frac{dR}{dt}\frac{\Lambda(R)}{R}\left[\nu-\sigma(R)\right] \nonumber\\
\ \nonumber\\
& = & -f(1-R)\frac{\Lambda(R)}{R}\left[\nu-\sigma(R)\right]=-{{\Lambda}^{2}}(R)\left[\nu-\sigma(R)\right].
\end{eqnarray*}
    It can be easily transformed into the form
\[
    d\left[\frac{1}{\Lambda(R)}-\nu{t}\right]=-\sigma(R)dt.
\]
    Integrating the last equality over $[0,t)$ leads to the equation
\begin{equation}                              \label{IIB3.11}
    \frac{1}{\Lambda (R(t;s))}-\frac{1}{\Lambda (1-s)}=\nu{t}-\int_{0}^{t}{\sigma\left(R(\tau;s)\right)d\tau}.
\end{equation}
    Relation {(\ref{IIB3.10})} implies that integral
\[
    \int_{0}^{t}{\sigma\left(R(\tau;s)\right)d\tau}=o(t) \qquad {\text{as}} \quad {x\to\infty}.
\]
    To more deeply interpret this integral, we once again employ the backward Kolmogorov equation ${\partial{F}}/{\partial{t}}=f(F)$
    together with the structural condition $[{f_{\nu}}]$. This leads to the following relations:
\begin{eqnarray}                             \label{IIB3.12}
    {\int_{0}^{t}{\sigma \left( R(\tau ;s) \right)d\tau }}
& = & \int_{0}^{t}{\sigma (R)\frac{d\tau }{dR}dR} \nonumber\\
\ \nonumber\\
& = & -\int_{0}^{t}{\frac{\sigma(R)}{f(1-R)}dR}=-\int_{1-s}^{R(t;s)}{\frac{\sigma(y)}{y\Lambda(y)}dy}.
\end{eqnarray}

    Next, taking into account the relation {(\ref{IIB3.10})}  and substituting $x={1}/{y}$ into the integrand on the right-hand
    side of {(\ref{IIB3.12})}, while recalling that $\tau (t;s):={1}/{R(t;s)}$, we write the following equality:
\begin{equation}                              \label{IIB3.13}
    \int_{0}^{t}{\sigma\left(R(\tau;s)\right)d\tau}=
    \int_{\tau(0;s)}^{\tau(t;s)}{\frac{\varepsilon(x)}{{x^{1-\nu}}\mathcal{L}(x)}dx},
\end{equation}
    where $\varepsilon(x)=\omega(x)\left(1+\mathcal{O}\left(\omega(x)\right)\right)$ as $x\to\infty$. We
    see that $\mathcal{M}(0)=0$ and
\[
    \frac{\varepsilon(x)}{{x^{1-\nu}}\mathcal{L}(x)}=\mathcal{O}\left(1/x\right)
    \qquad {\text{as}} \quad {x\to\infty}.
\]
    Then it follows from {(\ref{IIB3.1})} and {(\ref{IIB3.13})} that
\begin{equation}                              \label{IIB3.14}
    \int_{\tau(0;s)}^{\tau(t;s)}{\frac{\varepsilon(x)}{{x^{1-\nu}}\mathcal{L}(x)}dx}=
    \int_{\tau(0;s)}^{\tau(t)}{\frac{\varepsilon(x)}{{{x}^{1-\nu}}\mathcal{L}(x)}dx}+\mathcal{O}\left(1/t\right)
\end{equation}
    as $t\to\infty$, where $\tau(t)=\tau(t;0)$. Now relation {(\ref{IIB3.4})}
    directly follows from combination of {(\ref{IIB3.11})}--{(\ref{IIB3.14})}.

    Further, assumption {(\ref{IIB3.3})} implies that ${x^{\nu}}\omega(x)\sim\mathcal{L}(x)$ as $x\to\infty$.
    Then it follows from relation {(\ref{IIB3.10})} that
\begin{equation}                              \label{IIB3.15}
    \frac{\sigma(R(t;s))}{\Lambda(R(t;s))}=1+\mathcal{O}\left(\Lambda(R(t;s))\right)
    \qquad {\text{as}} \quad {t\to\infty}.
\end{equation}
    Since $R(t;s)\to{0}$ as $t\to\infty$ uniformly in $s\in[0,d],\;d<1$, we can relation {(\ref{IIB3.15})} implies that
    the function ${{\mathcal{L}}_{\sigma}}(y):={\sigma(y)}/{\Lambda(y)}$ belongs to the class of Karamata functions
    ${{\mathcal{SV}}_{0}}(\Lambda)$. Therefore, denoting ${{\mathcal{L}}_{\omega}}(x):={{\mathcal{L}}_{\sigma}}\left(1/x\right)$
    and combining {(\ref{IIB3.12})} with {(\ref{IIB3.15})}, we obtain the relations
\begin{eqnarray}                              \label{IIB3.16}
    {\int_{0}^{t}{\sigma\left(R(\tau;s)\right)d\tau}}
& = & -\int_{0}^{t}{\frac{\sigma(R)}{R\Lambda(R)}dR}=-\int_{0}^{t}{\frac{{{\mathcal{L}}_{\sigma}}(R)}{R}dR} \nonumber\\
\ \nonumber\\
& = & -\int_{1-s}^{R(t;s)}{\frac{{{\mathcal{L}}_{\sigma }}(y)}{y}dy}
        =\int_{\tau(0;s)}^{\tau(t;s)}{\frac{{{\mathcal{L}}_{\omega }}(x)}{x}dx},
\end{eqnarray}
    where ${{\mathcal{L}}_{\omega}}(\cdot)\in{{\mathcal{SV}}_{\infty}}(\omega)$ such that ${{\mathcal{L}}_{\omega}}(t)\to{1}$
    as $t\to\infty$. In the final step of formula {(\ref{IIB3.16})}, we replaced $x$ with ${1}/{y}$. Next,
    substituting $x=u\tau(t;s)$ into the integrand on the right-hand-side of {(\ref{IIB3.16})} and by applying
    the methods from {\cite{ImMDPI24}}, we can easily verify that
\begin{equation}                              \label{IIB3.17}
    \int_{\tau(0;s)}^{\tau(t;s)}{\frac{{{\mathcal{L}}_{\omega}}(x)}{x}dx}={{\mathcal{L}}_{\omega}}(\tau(t;s))
    \ln\frac{\tau(t;s)}{\tau(0;s)}-\int_{T(t;s)}^{1}{\frac{1}{u}\left[1-
    \frac{{{\mathcal{L}}_{\omega}}(u\tau(t))}{{{\mathcal{L}}_{\omega}}(\tau(t))}\right]du},
\end{equation}
    where $T(t;s)={\tau(t;0)}\left/{\tau(t;s)}\right.$. Since $\tau(t;s)\to\infty$ as $t\to\infty$ and the function
    ${{\mathcal{L}}_{\omega}}(\cdot)\in{{\mathcal{SV}}_{\infty}}(\omega)$, the term in brackets on
    the right-hand side of {(\ref{IIB3.17})} decays to zero at the rate
    of $\left.\mathcal{O}\left({\mathcal{L}(\tau(t))}\right/{{\tau}^{\nu}(t)}\right)$.
    Then relation {(\ref{IIB3.17})} becomes
\begin{equation}                             \label{IIB3.18}
    \int_{\tau(0;s)}^{\tau(t;s)}{\frac{{{\mathcal{L}}_{\omega}}(x)}{x}dx}=
    {{\mathcal{L}}_{\omega}}(\tau(t;s))\ln\frac{\tau(t;s)}{\tau(0;s)}\left(1+o(1)\right)
    \qquad {\text{as}} \quad {t\to\infty}.
\end{equation}
    Finally, according to {\cite[Th3]{ImMDPI24}}, we confirm that
\[
    {{\mathcal{L}}_{\omega}}(u)=1+\mathcal{O}\left({1}/{{{u}^{\nu }}}\right)
    \qquad {\text{as}} \quad {u\to\infty}.
\]
    Now, by combining formulas {(\ref{IIB3.11})}, {(\ref{IIB3.16})},
    and {(\ref{IIB3.18})}, we achieve relation {(\ref{IIB3.5})}.

    The lemma is proved.
\end{proof}

    The {\hyperref[IIBLem1]{Basic~Lemma}~\ref{IIBLem1}} plays a pivotal role in the context of
    our research. Its fundamental significance is primarily due to the fact that the expression
\[
    q(t):=R(t;0)=\mathbb{P}\left\{Z(t)\in\mathcal{S}\right\}
\]
    precisely determines the survival probability of MBS initiated by a single founder. In this framework,
    formula {(\ref{IIB3.5})} provides an asymptotic representation of this probability under Basic Axioms.
    We have the following result.
\begin{theorem}                      \label{IIBTh1}
    Let conditions $[{{f}_{\nu }}]$, $[{{\mathcal{L}}_{\nu }}]$ and $[{{\omega }_{\nu }}]$ be satisfied.
    Then, the survival probability $q(t)$ has the following asymptotic representation:
\begin{equation}                             \label{IIB3.19}
    q(t)=\frac{\mathcal{N}(t)}{{(\nu{t})}^{1/{\nu}}}\left(1+\frac{\ln{{a}_{0}}\nu{t}}{{{\nu}^{3}}t}
    \left(1+o(1)\right)\right) \qquad {\text{as}} \quad {t\to\infty},
\end{equation}
    where the function $\mathcal{N}(\cdot)\in{{\mathcal{SV}}_{\infty}}(\omega)$ and it has the property {(\ref{IIB3.2})}.
\end{theorem}

\begin{proof}
    The proof follows immediately from {(\ref{IIB3.5})}.
\end{proof}

    Statements of {\hyperref[IIBLem1]{Lemma}~\ref{IIBLem1}} and {\hyperref[IIBTh1]{Theorem}~\ref{IIBTh1}} provide
    a significant refinement of the asymptotic behavior of  both ${R(t; s)}$ and ${q(t)}$, offering a unified and
    more precise description than previously available. This result not only consolidates but also extends the
    findings of earlier works such as {\cite{Pakes2010}} and {\cite{Zol57}}, establishing a new benchmark
    in the asymptotic analysis of the system.

    In addition to the function $R(t;s)$, the asymptotic expansion of its derivative
    ${R}'(t;s):={\partial{R(t;s)}}\big/{\partial{s}}$ is of considerable importance. This is particularly justified
    by the identity ${R}'(t;0)=-{\textsf{p}_1}(t)$, where ${\textsf{p}_1}(t)$ denotes the probability that the system
    will again be in state $1$ within time $t$. A detailed analysis of this derivative not only reflects the local
    growth behavior of the system but also serves as a foundational step toward deriving asymptotic expansions for
    the transition probabilities ${{P}_{ij}}(t):={{\mathbb{P}}_{i}}\left\{Z(t)=j\right\}$ for all $i,j\in\mathcal{S}$.
    Thus, the study of ${R}'(t;s)$ offers valuable insights into the fine probabilistic structure of the branching dynamics.

\begin{lemma}                  \label{IIBLem2}
    Under the conditions $[{f}_{\nu}]$, $[{\mathcal{L}}_{\nu}]$ and $[{\omega}_{\nu}]$, the following relation holds:
\begin{equation}                             \label{IIB3.20}
    \frac{\partial{R(t;s)}}{\partial{s}}=-\frac{R(t;s)}{f(s)}\frac{1}{\nu{t}}
    \left(1+\frac{\ln\left[\Lambda(1-s)\nu{t}\right]}{{{\nu}^2}t}+r(t;s)\right),
\end{equation}
    where $r(t;s)=o\left({\ln{t}}/{t}\right)$ uniformly in $s\in[0,1)$ as $t\to\infty$.
\end{lemma}

\begin{proof}
    Combination of equations {(\ref{IIB2.2})} and {(\ref{IIB2.3})} with $[{f}_{\nu}]$, gives
\begin{equation}                  \label{IIB3.21}
    \frac{\partial{R(t;s)}}{\partial{s}}=-\frac{R(t;s)}{f(s)}\Lambda\left(R(t;s)\right).
\end{equation}
    To seek an expression for the function $\Lambda(R(t;s))$, we refer to relation {(\ref{IIB3.5})}.
    Substituting it into the equation {(\ref{IIB3.21})} leads directly to the desired result {(\ref{IIB3.20})}.
\end{proof}

    The next theorem, which establishes the asymptotic relation between the local probability ${\textsf{p}_1}(t)$
    and the survival probability $q(t)$, follows as an immediate consequence of {\hyperref[IIBLem2]{Lemma}~\ref{IIBLem2}}.

\begin{theorem}                      \label{IIBTh2}
    Let conditions $[{f}_{\nu}]$, $[{\mathcal{L}}_{\nu}]$ and $[{\omega}_{\nu}]$ be satisfied. Then
\begin{equation}                  \label{IIB3.22}
    \frac{{\textsf{p}_1}(t)}{q(t)}=\frac{1}{{a_{0}}\nu{t}}\left(1+\frac{\ln{a_{0}}\nu{t}}{{{\nu}^{2}}t}
    \left(1+o(1)\right)\right) \qquad {\text{as}} \quad {t\to\infty}.
\end{equation}
\end{theorem}

    The statements of last two theorems imply that ${{(\nu{t})}^{1+{1}/{\nu}}}{\textsf{p}_1}(t)$ belongs
    to the class of slowly varying functions at infinity with the remainder, and we can write the following
\begin{theorem}                  \label{IIBTh3}
    Under conditions of {\hyperref[IIBTh2]{Theorem}~\ref{IIBTh2}}
\begin{equation}                \label{IIB3.23}
    {{(\nu{t})}^{1+{1}/{\nu}}}{\textsf{p}_1}(t)=\frac{1}{{{a}_{0}}}{{\mathcal{N}}_{\lambda}}(t),
\end{equation}
    where ${{\mathcal{N}}_{\lambda}}(t)\in{{\mathcal{SV}}_{\infty }}(\omega)$ and
    ${{{\mathcal{N}}_{\lambda}}(t)}\big/{\mathcal{N}(t)}=1+\mathcal{O}\left( {\ln{t}}\big/{t}\right)$ as $t\to\infty$.
\end{theorem}

\begin{remark}                  \label{IIBRem2}
    By applying the monotone ratio convergence property {\cite{ImIJSA14}}, we can we readily confirm that probability
    functions ${\textsf{p}_1}(t)\in{{\mathcal{SV}}_{\infty}}(\omega)$. Moreover, this property allows us to
    derive a precise asymptotic expansion for these probability functions.
\end{remark}

    We conclude this section by formulating some characteristic property of the function $\mathcal{L}(\cdot)$,
    which reveals the regularity of its behavior for a given functional. This property, previously established
    in {\cite{ImCCIS24}}, not only reflects the intrinsic structure of $\mathcal{L}(\cdot)$, but also plays a significant
    role in the analysis conducted in the subsequent sections. The exact formulation of this result is given below as a lemma.

\begin{lemma}                   \label{IIBLem3}
    Let $K(y)$ be a positive function defined on $y\in(0,\infty)$, such that $K(y)\downarrow{0}$ as
    $y\downarrow{0}$, and define $\phi(y):=y-yK(y)$. If conditions $[{f}_{\nu}]$, $[{\mathcal{L}}_{\nu}]$ and
    $[{\omega}_{\nu}]$ hold, then the following asymptotic relation holds as $y\downarrow{0}$:
\begin{equation}                \label{IIB3.24}
    \mathcal{L}\left(\frac{1}{\phi(y)}\right)=\mathcal{L}\left(\frac{1}{y}\right)\left(1+K(y)\omega\left(\frac{1}{y}\right)\right).
\end{equation}
\end{lemma}

\section{\textbf{Convergence Theorems and Invariant Measures}}             \label{IIBSec4}

    First, we observe invariant properties of the transient MBIS. Similar to the case of the
    remainder term $\omega(x)$ for $\mathcal{L}(x)$, it is entirely justified to further refine
    our analysis by narrowing our focus to the following more precise and specific asymptotic regime:
\[
    \frac{\psi(x)}{\ell(x)}{x^{\delta}}\to{1} \qquad {\text{as}} \quad {x\to\infty}
\]
    which entails that
\begin{equation}                \label{IIB4.1}
    {{\lambda}_{\psi}}:=\lim_{x\to\infty }{x^{\delta}}\psi(x)<\infty.
\end{equation}
    As a preliminary step, we present the following important results, which directly follow from the corresponding
    theorems in {\cite{ImMDPI24}} and will serve as the foundation for our forthcoming discussions and subsequent analysis.

\begin{lemma}                   \label{IIBLem4}
    Assume that conditions {(\ref{IIB3.3})} and {(\ref{IIB4.1})} hold. Then, the following relations are satisfied:
\begin{align}                \label{IIB4.2}
\left\{
\begin{array}{l}
    \displaystyle{{\lambda}_{\omega}}-\mathcal{L}(t)=\frac{{{\lambda}_{\omega}}}{\nu{t^{\nu}}}{L_{\omega}}(t)
    \qquad \hfill {\text{under assumption {(\ref{IIB3.3})},}}
\vspace{2.8mm}  \\
    \displaystyle{{\lambda}_{\psi}}-\ell(t)=\frac{{{\lambda}_{\psi}}}{\delta{t^{\delta}}}{L_{\psi}}(t)
    \qquad \hfill {\text{under assumption {(\ref{IIB4.1})},}}  \\
\end{array}
\right.
\end{align}
    where ${{L}_{\omega}}(t)=1+\mathcal{O}(\omega(t))$ and ${{L}_{\psi}}(t)=1+\mathcal{O}(\psi(t))$ as $t\to\infty$.
\end{lemma}

\begin{lemma}                   \label{IIBLem5}
    Under the conditions of {\hyperref[IIBLem4]{Lemma}~\ref{IIBLem4}} the ration function
\[
    \textsf{L}(t):=\frac{\ell(t)}{\mathcal{L}(t)}
\]
    belongs to the class $\textsf{L}(\cdot)\in{{\mathcal{SV}}_{\infty}}(\psi)$ and  satisfies the asymptotic expansion
\begin{equation}                \label{IIB4.3}
    {C_{\textsf{L}}}-{\textsf{L}}(t)=\frac{{C_{\textsf{L}}}}{\delta{t^{\delta}}}+\mathcal{O}\left(\frac{1}{t^{\nu}}\right)
    \qquad {\text{as}} \quad {t\to\infty},
\end{equation}
    where ${C_{\textsf{L}}}={{\lambda}_{\psi}}\big/{{\lambda}_{\omega}}$.
\end{lemma}

    Further, let as before $\tau(t):=\tau(t;0)$ and $\mathcal{T}(t):={{[\tau(t)]}^{|\gamma|}}$. The results
    proved in {\cite{ImUfa21}} imply that if the MBIS state space $\mathcal{S}$ is transient, $\gamma<0$, then
\[
    {{[\tau(t)]}^{-|\gamma|}}\ln{{p}_{00}}(t)\in{{\mathcal{SV}}_{\infty}}(\psi ).
\]
    This observation indicates that for $\gamma<0$, the main emphasis should be placed on investigating the asymptotic
    behavior of the scaled GF ${e^{\mathcal{T}(t)}}\mathcal{P}(t;s)$ rather than the original $\mathcal{P}(t;s)$.
    In particular, in the fully admissible but exceptional case where the ratio function $\textsf{L}(t)$ converges
    to $|\gamma|$ as $t\to\infty$, it becomes possible to derive an explicit representation of the limiting-invariant GF
\[
    \mathcal{U}(s):=\lim_{t\to\infty}{{e}^{\mathcal{T}(t)}}\mathcal{P}(t;s).
\]
    As shown below, the structure of $\mathcal{U}(s)$ closely reflects the form of equation {(\ref{IIB2.6})}, yielding an
    ``excellent result'' characterized by both its explicit form and analytical clarity. The following theorem rigorously
    establishes the explicit form of $\mathcal{U}(s)$) and provides a precise characterization of the rate at which
    ${e^{\mathcal{T}(t)}}\mathcal{P}(t;s)$ converges to this function, offering deeper insight into the convergence behavior.

    The following theorem rigorously establishes the explicit form of the limit function $\mathcal{U}(s)$,
    and provides a precise characterization of the rate at which ${e^{\mathcal{T}(t)}}\mathcal{P}(t;s)$
    converges to $\mathcal{U}(s)$, offering deeper insight into the convergence behavior.

\begin{theorem}                      \label{IIBTh4}
    Let Basic Axioms hold with $\gamma<0$ and let $\mu:=2\delta-\nu>0$. If ${C_{\textsf{L}}}=|\gamma|$
    in {(\ref{IIB4.3})}, then
\begin{equation}                \label{IIB4.4}
    {{e}^{\mathcal{T}(t)}}\mathcal{P}(t;s)=\mathcal{U}(s)\cdot\left(1+\zeta(t;s)\right),
\end{equation}
    where the limiting GF $\mathcal{U}(s)$ has the form
\begin{equation}                \label{IIB4.5}
    \mathcal{U}(s)=\exp\left\{\frac{1}{{(1-s)}^{|\gamma|}}+
    \int_{{1}/{(1-s)}}^{\infty}{\frac{|\gamma|-\textsf{L}(u)}{{u^{1-|\gamma|}}}du}\right\}
\end{equation}
    and
\[
    \zeta(t;s)=\frac{|\gamma|}{\delta\mu}\frac{{{\mathcal{N}}^{\mu}}(t)}{{(\nu{t})}^{{\mu}/{\nu}}}\left(1+\sigma(t;s)\right),
\]
    where $\mathcal{N}(\cdot)\in{{\mathcal{SV}}_{\infty}}$ is defined with the property {\eqref{IIB3.2}}
    and $\sigma(t;s)=\mathcal{O}\left({\ln{t}}/{t}\right)$ as $t\to\infty$ uniformly in $s\in[0,d],\;d<1$.
\end{theorem}

\begin{proof}
    Using the backward Kolmogorov equation {\eqref{IIB2.2}} in {\eqref{IIB2.1}}, we have
\begin{eqnarray}                              \label{IIB4.6}
    {{{e}^{\mathcal{T}(t)}}\mathcal{P}(t;s)}
& = & \exp\left\{{{\left[\tau(t)\right]}^{|\gamma|}}+\int_{0}^{t}{h\left(F(u;s)\right)du}\right\} \nonumber\\
\ \nonumber\\
& = & \exp \left\{{{\left[\tau(t;s)\right]}^{|\gamma|}}+\int_{s}^{F(t;s)}{\frac{h(u)}{f(u)}du}+\Delta(t;s)\right\},
\end{eqnarray}
    where $\Delta(t;s)={{[\tau(t)]}^{|\gamma|}}-{{[\tau(t;s)]}^{|\gamma|}}$. It is easy to see that
\[
    {{\left[\tau(t;s)\right]}^{|\gamma|}}=\frac{1}{{\left( 1-F(t;s) \right)}^{|\gamma|}}
    =\frac{1}{{(1-s)}^{|\gamma|}}+\int_{s}^{F(t;s)}{\frac{|\gamma|}{{(1-u)}^{1+|\gamma|}}du}.
\]
    Then we can rewrite {(\ref{IIB4.6})} as follows:
\begin{equation}                \label{IIB4.7}
    {{e}^{\mathcal{T}(t)}}\mathcal{P}(t;s)=\exp\left\{\frac{1}{{(1-s)}^{|\gamma|}}+\mathcal{B}(t;s)+\Delta (t;s)\right\},
\end{equation}
    where
\[
    \mathcal{B}(t;s)=\int_{s}^{F(t;s)}{\left[\frac{h(u)}{f(u)}+\frac{|\gamma|}{{(1-u)}^{1+|\gamma|}}\right]du}.
\]

    First we estimate $\Delta (t;s)$. According to our notation and
    using the asymptotic formula {\eqref{IIB3.5}}, we have
\begin{equation}                \label{IIB4.8}
    {{\left[\tau(t;s)\right]}^{|\gamma|}}={{\left[\nu{t}\mathcal{L}\left(\frac{1}{R(t;s)}\right)\right]}^{{|\gamma|}/{\nu}}}
    \left[1-\frac{|\gamma|{{\lambda}_{\omega}}}{{\nu}^{3}}\frac{\ln[\Lambda(1-s)\nu{t}]}{t}\left(1+o(1)\right)\right]
\end{equation}
    as $t\to\infty$. At the same time, from the previously established
    formula {\eqref{IIB3.1}}, we can immediately write that
\[
    R(t;s)=q(t)-q(t)\mathcal{M}(t;s),
\]
    where $q(t)=R(t;0)$ as before and $\nu{t}\mathcal{M}(t;s)\to\mathcal{M}(s)$ as $t\to\infty$. Therefore,
    since $q(t)\downarrow{0}$ as $t\to\infty$, by utilizing formula {\eqref{IIB3.24}} for $\phi(y)=R(t;s)$
    with $y=q(t)$ and $K(y)=\mathcal{M}(t;s)$, we can deduce that
\begin{equation}                \label{IIB4.9}
    \mathcal{L}\left(\frac{1}{R(t;s)}\right)=\mathcal{L}\left(\frac{1}{q(t)}\right)
    \left[1+\mathcal{M}(t;s)\omega\left(\frac{1}{q(t)}\right)\right].
\end{equation}
    Given our assumptions and according to asymptotic representation {\eqref{IIB3.19}} and
    first assertion of {\eqref{IIB4.2}}, it immediately becomes obvious that
\begin{align}                \label{IIB4.10}
\left\{
\begin{array}{l}
    \omega\left({1}\big/{q(t)}\right)=\mathcal{O}\left({q^{\nu}}(t)\right)=\mathcal{O}\left({1}\big/{t}\right)
    \qquad \hfill {\text{as}} \quad {t\to\infty,}
\vspace{2.8mm}  \\
    \mathcal{L}\left({1}\big/{q(t)}\right)={{\lambda}_{\omega}}+\mathcal{O}\left({1}\big/{t}\right)
    \qquad \hfill {\text{as}} \quad {t\to\infty.}\\
\end{array}
\right.
\end{align}
    Therefore, the utilization of relations {\eqref{IIB4.9}} and {\eqref{IIB4.10}} transforms equality {\eqref{IIB4.8}} into the form
\[
    {{\left[\tau(t;s)\right]}^{|\gamma|}}={{\left[\nu{t}\mathcal{L}\left(\frac{1}{q(t)}\right)\right]}^{{|\gamma|}/{\nu}}}
    \left[1-\frac{|\gamma|{{\lambda}_{\omega}}}{{\nu}^{3}}\frac{\ln[\Lambda(1-s)\nu{t}]}{t}+
    \mathcal{O}\left(\frac{\mathcal{M}(s)}{{t^{2}}}\right)\right].
\]
    Hence
\begin{eqnarray}              \label{IIB4.11}
    {\Delta(t;s)}
& = & {{\left[\nu{t}\mathcal{L}\left(\frac{1}{q(t)}\right)\right]}^{{|\gamma|}/{\nu}}}
    \frac{|\gamma|{{\lambda}_{\omega}}}{{\nu}^{3}}\frac{\ln\left[{\Lambda(1-s)}/{a_0}\right]}{t}
    \left(1+\mathcal{O}\left(\frac{1}{t}\right)\right) \nonumber\\
\ \nonumber\\
& = & \frac{|\gamma|C_{\mathcal{L}}^{1-{\delta}/{\nu}}{{\lambda}_{\omega}}}{{\nu}^2}
    \frac{\ln\left[{\Lambda(1-s)}/{a_0}\right]}{{(\nu{t})}^{{\delta}/{\nu}}}\left(1+\mathcal{O}\left(\frac{1}{t}\right)\right)
\end{eqnarray}
    as $t\to\infty$.

    Next, employing assumptions $[{f_{\nu}}]$ and $[{h_{\delta}}]$, we write
\begin{equation}                \label{IIB4.12}
    \mathcal{B}(t;s)=\int_{s}^{F(t;s)}{\frac{|\gamma|-\textsf{L}\left({1}\big/{(1-u)}\right)}{{{u}^{1+|\gamma|}}}du}
    =:\mathcal{B}(s)-\mathcal{J}(t;s),
\end{equation}
    where
\begin{equation}                \label{IIB4.13}
    \mathcal{B}(s)=\int_{{1}\big/{(1-s)}}^{\infty}{\frac{|\gamma|-\textsf{L}(x)}{{x^{1-|\gamma|}}}dx}
\end{equation}
    and
\begin{equation}                \label{IIB4.14}
    \mathcal{J}(t;s)=\int_{\tau(t;s)}^{\infty}{\frac{|\gamma|-\textsf{L}(x)}{x^{1-|\gamma|}}dx}.
\end{equation}
    At this stage, when deriving equations {\eqref{IIB4.13}} and {\eqref{IIB4.14}}, we utilized the substitution $1-u={1}/{x}$
    to simplify the subsequent transformation. According to {\eqref{IIB4.3}}, the integrand in equation {\eqref{IIB4.14}}
    asymptotically behaves as ${|\gamma|}/{\delta{x^{1+\mu}}}$ if $x\to\infty$. Consequently, it becomes necessary
    to analyze the integral $\int_{\tau(t;s)}^{\infty}{\left[|\gamma|/{\delta{x^{1+\mu}}}\right]dx}$ as $t\to\infty$.
    Then certainly that
\begin{equation}                \label{IIB4.15}
    \mathcal{J}(t;s)=\frac{|\gamma|}{\delta\mu}{{R}^{\mu}}(t;s)\left(1+o(1)\right)
    \qquad {\text{as}} \quad {t\to\infty}.
\end{equation}
    On the other hand, {\hyperref[IIBLem1]{Lemma}~\ref{IIBLem1}} implies
\begin{equation}                \label{IIB4.16}
    R(t;s)=\frac{\mathcal{N}(t;s)}{{(\nu{t})}^{{1}/{\nu}}}
    \left(1+\frac{\ln\left[\Lambda(1-s)\nu{t}\right]}{{{\nu}^{3}}t}\left(1+o(1)\right)\right)
\end{equation}
    as $t\to\infty$, where according to {\eqref{IIB4.9}},
    $\mathcal{N}(t;s)=\mathcal{N}(t)\left(1+\mathcal{O}\left({1}/{t}\right)\right)$. Then
    $\mathcal{J}(t;s)=\mathcal{O}\left( {1}/{{t^{\mu/\nu}}}\right)$. It is evident that both expressions $\Delta (t;s)$
    and $\mathcal{J}(t;s)$ approach zero. But we can confirm that the leading part of the tail term in the approximation
    of ${e^{\mathcal{T}(t)}}\mathcal{P}(t;s)$ by the limit function $\mathcal{U}(s)$ is estimated by $\mathcal{J}(t;s)$,
    since there is no doubt that $\mathcal{O}\left({1}/{t^{\mu/\nu}}\right)$ approaches zero much less rapidly
    than $\mathcal{O}\left({1}/{t^{\delta/\nu}}\right)$ does it as $t\to\infty$. It follows from {\eqref{IIB4.15}}
    and {\eqref{IIB4.16}}, that
\begin{equation}                \label{IIB4.17}
    \mathcal{J}(t;s)=\frac{|\gamma|}{\delta\mu}\frac{{{\mathcal{N}}^{\mu}}(t)}{{(\nu{t})}^{{\mu}/{\nu}}}
    \left(1+\frac{\mu}{{\nu}^{3}}\frac{\ln\left[\Lambda(1-s)\nu{t}\right]}{t}\left(1+o(1)\right)\right)
\end{equation}
    as $t\to\infty$. Therefore, by combining {\eqref{IIB4.7}}, {\eqref{IIB4.12}}, {\eqref{IIB4.17}},
    we conclude that
\[
    {e^{\mathcal{T}(t)}}\mathcal{P}(t;s)=\mathcal{U}(s)\exp\left\{\mathcal{J}(t;s)\right\}.
\]
    It suffices to observe that ${e^x}\sim{1+x}$ as $x\to{0}$, thereby establishing the convergence
    in {\eqref{IIB4.4}}, accompanied by the explicit remainder estimate given in the form of $\zeta(t;s)$.

    The theorem is proved completely.
\end{proof}

    Invoking the primary branching property and Kolmogorov-Chapman equation leads to the functional
    equation $F(t+u;s)=F\left(t;F(u;s)\right)$. Using this equation in {\eqref{IIB2.1}} and considering
    the limiting behavior as $t\to\infty$, we have
\begin{equation}                \label{IIB4.18}
    \frac{\mathcal{U}(s)}{\mathcal{U}\left(F(t;s)\right)}=\mathcal{P}(t;s).
\end{equation}
    Represent the function $\mathcal{U}(s)$ by its power series expansion
    $\mathcal{U}(s)=\sum_{j\in\mathcal{E}}{{u_j}{s^j}}$. Then, formula {\eqref{IIB2.1}}
    and equation {\eqref{IIB4.18}} entails, that
\begin{eqnarray*}
    {\sum_{j\in\mathcal{E}}{{u_j}{s^j}}}
& = & \sum_{i\in \mathcal{E}}{{{u}_{i}}{{\left(F(t;s)\right)}^{i}}\mathcal{P}(t;s)} \nonumber\\
\ \nonumber\\
& = & \sum_{i\in\mathcal{E}}{{{u}_{i}}{{\mathcal{P}}_{i}}(t;s)}=
        \sum_{i\in\mathcal{E}}{{{u}_{i}}\sum_{j\in \mathcal{E}}{{{p}_{ij}}(t){{s}^{j}}}}=
        \sum_{j\in\mathcal{E}}{\sum_{i\in\mathcal{E}}{{{u}_{i}}{p_{ij}}(t){s^{j}}}}.
\end{eqnarray*}
    Matching the coefficients of ${{s}^{j}}$ on both sides of the equation leads to a recursive identity of functional form
\[
    {{u}_{j}}=\sum_{i\in \mathcal{E}}{{{u}_{i}}{{p}_{ij}}(t)},
\]
    which reveals an invariance property of $\left\{{u_j}\right\}$ with respect to transition functions ${p_{ij}}(t)$.

\begin{remark}                  \label{IIBRem3}
    In view of assertion {\eqref{IIB4.3}}, the integral appearing in the exponential term of {\eqref{IIB4.5}}
    simplifies to the representation
\begin{equation}                \label{IIB4.19}
    {{\mathcal{B}}_{\mu}}(s):=\int_{{1}/{(1-s)}}^{\infty}{\frac{|\gamma|-\textsf{L}(u)}{{u^{1-|\gamma|}}}du}
    =\frac{1}{\mu}{{\left(1-s\right)}^{\mu}}+\delta(s),
\end{equation}
    where $\delta(s)=\mathcal{O}{{(1-s)}^{1+\delta}}$ as $s\uparrow{1}$. Consequently, the limiting
    GF $\mathcal{U}(s)$ attains enhanced analytical clarity, thereby enabling a more precise and rigorous
    exposition of the structural characteristics underlying the invariant measure $\left\{{u_j}\right\}$.
\end{remark}

    As a widely applicable and highly effective alternative, the construction of an invariant measure can emerge
    quite naturally through the monotone convergence property applied to the ratio ${{p_{ij}}(t)}\big/{{p_{00}}(t)}$,
    reflecting the inherent regularity embedded in the structure of the transition dynamics. Since $F(t;s)\to{1}$
    as $t\to\infty$, it follows from {\eqref{IIB2.1}} that it suffices to focus on the case $i=0$. The following
    theorem rigorously establishes the structure of the limiting invariant measure, which is governed by
    the aforementioned ratio convergence property.

\begin{theorem}                      \label{IIBTh5}
    Under the assumptions of {\hyperref[IIBTh4]{Theorem}~\ref{IIBTh4}}, the nonnegative limits
\[
    \lim_{t\to\infty}\frac{{p_{0j}}(t)}{{p_{00}}(t)}={{\pi}_{j}}<\infty
\]
    exist and are generated by the limiting GF
\begin{equation}                \label{IIB4.20}
    \pi(s)=\sum_{j\in\mathcal{E}}{{{\pi}_{j}}{s^j}}=
    \exp\left\{\frac{1}{{(1-s)}^{|\gamma|}}{\textsf{L}_{v}}\left(\frac{1}{1-s}\right)\right\},
\end{equation}
    where ${\textsf{L}_{v}}(\cdot )\in{{\mathcal{SV}}_{\infty}}(\psi)$ and
    ${{\textsf{L}_{v}}(t)}\big/{\textsf{L}(t)}\to{{1}\big/{|\gamma|}}$ as $t\to\infty$.
\end{theorem}

\begin{proof}
    Referring to {\eqref{IIB2.1}} we write
\begin{eqnarray}                \label{IIB4.21}
    {\pi(t;s)}
& := & \sum_{j\in\mathcal{E}}{\frac{{p_{0j}}(t)}{{p_{00}}(t)}{s^{j}}}
        =\frac{\mathcal{P}(t;s)}{\mathcal{P}(t;0)} \nonumber\\
\ \nonumber\\
& = & \exp\left\{\int_{0}^{t}{\left[h\left(F(u;s)\right)-h\left(F(u;0)\right)\right]du}\right\}.
\end{eqnarray}
    First, since $F(u+\tau;0)=F(u;F(\tau;0))$, denoting $s=F(\tau;0)$, we write
\[
    \int_{0}^{t}{h\left(F(u;s)\right)du}=\int_{0}^{t}{h\left(F(u+\tau;0)\right)du}
    =\int_{\tau}^{\tau+t}{h\left(F(u;0)\right)du}.
\]
    Combining this identity with the classical properties of the Riemann integral allows for a significant
    simplification of the expression in exponent of formula {\eqref{IIB4.21}}, ultimately yielding the following identity.
\begin{eqnarray*}
    {\int_{0}^{t}{\left[h(F(u;s))-h(F(u;0))\right]du}}
& = & \int_{t}^{t+\tau}{h(F(u;0))du}-\int_{0}^{\tau}{h(F(u;0))du} \nonumber\\
\ \nonumber\\
& = & \int_{0}^{\tau}{\left[h(F(u+t;0))-h(F(u;0))\right]du}.
\end{eqnarray*}
    Therefore,
\begin{equation}                \label{IIB4.22}
    \int_{0}^{t}{\left[ h(F(u;s))-h(F(u;0))\right]du} \to
    -\int_{0}^{s}{\frac{h(y)}{f(y)}dy}  \qquad {\text{as}} \quad {t\to\infty}.
\end{equation}
    At this point, we take into account that $h(F(t;0))$ vanishes as $t\to\infty$, and apply the backward
    Kolmogorov equation {\eqref{IIB2.2}}. These observations, when combined with formulas {\eqref{IIB4.21}}
    and {\eqref{IIB4.22}}, immediately yield that
\[
    \pi(t;s)\to \pi(s)=\exp\left\{-\int_{0}^{s}{\frac{h(y)}{f(y)}dy}\right\}
    \qquad {\text{as}} \quad {t\to\infty}.
\]
    Nevertheless, under the combined assumptions $[{f_{\nu}}]$ and $[{h_{\delta}}]$, it necessarily follows
    that the asymptotic behavior of the system is governed by the interplay between the rates of variation of
    the functions $f(s)$ and $h(s)$, ultimately leading to the following form of the generating function $\pi(s)$:
\begin{equation}                \label{IIB4.23}
    \pi(s)=\exp\left\{\int_{1}^{{1}/{(1-s)}}{{u^{-(1-|\gamma|)}}\textsf{L}(u)du}\right\}.
\end{equation}
    It is now observed that the integral appearing in the exponent in {\eqref{IIB4.23}} falls naturally
    within the scope of the formula established in {\cite[Th5]{ImMDPI24}}. In accordance with this result, the
    following representation is obtained:
\begin{equation}                \label{IIB4.24}
    \int_{1}^{{1}/{(1-s)}}{{u^{-(1-|\gamma|)}}\textsf{L}(u)du}=\frac{1}{|\gamma|}\frac{1}{{(1-s)}^{|\gamma|}}
    \textsf{L}\left(\frac{1}{1-s}\right)\left(1+\mathcal{O}{{(1-s)}^{\beta}} \right)
\end{equation}
    as $t\to\infty$, where $\beta=\min\left\{\delta,|\gamma|\right\}$. Now, denoting
\[
    {\textsf{L}_v}(t):=\frac{1}{|\gamma|}\textsf{L}(t)\left(1+\mathcal{O}{{\left(\frac{1}{t^{\beta}}\right)}}\right),
\]
    and combining relations {\eqref{IIB4.21}}--{\eqref{IIB4.24}},
    we obtain the representation given in {\eqref{IIB4.20}}.

    The theorem is proved.
\end{proof}

\begin{remark}                  \label{IIBRem4}
    A direct comparison of {\hyperref[IIBTh4]{Theorem}~\ref{IIBTh4}} and {\hyperref[IIBTh5]{Theorem}~\ref{IIBTh5}}
    confirms the identity $\mathcal{U}(s)=\mathcal{U}(0)\pi(s)$, which rigorously establishes that the two invariant
    measures $\left\{{u_{j}}\right\}$ and $\left\{{{\pi}_{j}}\right\}$ are not only equivalent up to a multiplicative
    constant
\[
    {{u}_{0}}={\lim_{t\to\infty}}{{e}^{\mathcal{T}(t)}}{p_{00}}(t)
\]
    but, in fact, represent the unique invariant measure (up to normalization) associated
    with the long-term dynamics of the system.
 \end{remark}

    As established in {\hyperref[IIBTh5]{Theorem}~\ref{IIBTh5}}, the limiting behavior of the ratio $p_{0j}(t)/p_{00}(t)$
    plays a pivotal role in the asymptotic analysis of the transition functions. In this context, it becomes both natural
    and essential to derive an explicit asymptotic representation of the expression $e^{\mathcal{T}(t)}p_{00}(t)$.
    This serves as a universal scaling factor in the limiting form of $p_{0j}(t)$, and its precise structure is
    crucial for identifying the exact form and the uniqueness of the invariant measure associated with the MBIS.
    The necessity of investigating this function arises directly from the asymptotic framework developed
    in {\hyperref[IIBTh5]{Theorem}~\ref{IIBTh5}} and completes the foundation for a comprehensive
    description of the systems long-term behavior.

\begin{theorem}                      \label{IIBTh6}
    Let the conditions of {\hyperref[IIBTh4]{Theorem}~\ref{IIBTh4}} be satisfied. Then
    ${{\mathcal{B}}_{\mu}}:={{\mathcal{B}}_{\mu}}(0)$ is finite, where ${{\mathcal{B}}_{\mu}}(s)$ is
    given in {\eqref{IIB4.19}}, and
\begin{equation}                \label{IIB4.25}
    {{e}^{\mathcal{T}(t)}}{{p}_{00}}(t)={{u}_{0}}\cdot \left(1-{{\mathcal{J}}_{\mu}}(t)\right),
\end{equation}
    where ${{u}_{0}}=\exp\left\{1+{{\mathcal{B}}_{\mu}}\right\}$ and
\[
    {{\mathcal{J}}_{\mu }}(t)=\frac{1}{\delta\mu}\frac{{{\mathcal{N}}^{\mu}}(t)}{{(\nu{t})}^{{\mu}/{\nu}}}
    \left(1+\frac{\mu\ln{{a}_{0}}\nu{t}}{{{\nu}^{3}}t}\left(1+o(1)\right)\right) \qquad {\text{as}} \quad {t\to\infty}.
\]
\end{theorem}

\begin{proof}
    First, by setting $s=0$ in {\eqref{IIB4.7}} and combining with relations {\eqref{IIB4.12}}--{\eqref{IIB4.14}},
    while accounting for the fact that $\Delta (t;s)=o\left(\mathcal{J}(t;s)\right)$, directly yields the asymptotic
    relation
\begin{equation}                \label{IIB4.26}
    {e^{\mathcal{T}(t)}}{p_{00}}(t)=
    \exp\left\{1+{{\mathcal{B}}_{\mu}}-{{\mathcal{J}}_{\mu}}(t)\right\}\left(1+o(1)\right),
\end{equation}
    which accurately captures the leading-order behavior of the expression as $t\to\infty$, where
\[
    {{\mathcal{B}}_{\mu}}=\int_{1}^{\infty}{\frac{|\gamma|-\textsf{L}(u)}{{{u}^{1-|\gamma|}}}du}
\]
    and ${{\mathcal{J}}_{\mu }}(t)=\mathcal{J}(t;0)$. Invoking assertion {\eqref{IIB4.3}}, we observe that
\begin{equation}                \label{IIB4.27}
    {{\mathcal{J}}_{\mu}}(t)=\int_{\tau(t)}^{\infty}{\frac{|\gamma|-\textsf{L}(x)}{{x^{1-|\gamma|}}}dx}
    =\frac{1}{\delta \mu }{q^{\mu}}(t)+\delta(t),
\end{equation}
    where $q(t)$ is estimated in {\eqref{IIB3.19}} and $\delta(t)=\mathcal{O}\left({{q}^{\delta}}(t)\right)$
    as $t\to\infty$. Similarly, it can be directly verified that ${{\mathcal{B}}_{\mu}}<\infty$. Consequently,
    expansion {\eqref{IIB4.25}} with remainder term ${{\mathcal{J}}_{\mu}}(t)$ readily follows from the
    combination of {\eqref{IIB3.19}}, {\eqref{IIB4.26}}, and {\eqref{IIB4.27}}.

    The theorem is proved.
\end{proof}

    Now, returning to MBS without immigration, denoted by $\left\{Z(t)\right\}$, we introduce the conditioned
    probability functions $\textsf{p}_j^{\mathcal{S}}(t):=\mathbb{P}\left\{Z(t)=j\mid{Z(t)}\in\mathcal{S}\right\}$,
    and define the associated conditional GF by
\[
    {{\mathcal{P}}^{\mathcal{S}}}(t;s):=\sum\limits_{j\in \mathcal{S}}{\textsf{p}_{j}^{\mathcal{S}}(t){{s}^{j}}}
\]
    which  compactly encodes the time evolution of the system within the restricted state space $\mathcal{S}$.
    This conditioning naturally gives rise to a new population system $\left\{{{Z}^{\mathcal{S}}}(t)\right\}$,
    which we refer to as the positively conditioned branching system. It is constructed from the original
    system  $\left\{Z(t)\right\}$ by conditioning trajectories to remain in $\mathcal{S}$. As such,
    $\left\{{{Z}^{\mathcal{S}}}(t)\right\}$ is a time-homogeneous, irreducible continuous-time Markov chain with
    strictly positive trajectories. The transition dynamics of $\left\{{{Z}^{\mathcal{S}}}(t)\right\}$ reflect
    a probabilistic restructuring of the original system under the positivity constraint. Such constructions are
    conceptually related to Doob's $h$-transform and to the so-called $Q$-processes, wherein the dynamics are
    altered by conditioning on survival or restriction to a subset of states. In our context, this framework
    plays a pivotal role in understanding long-term behavior and in identifying invariant measures for the
    system. The GF ${{\mathcal{P}}^{\mathcal{S}}}(t;s)$ provides a compact representation of the distribution
    of the system $\left\{{{Z}^{\mathcal{S}}}(t)\right\}$ over the subset $\mathcal{S}$, conditioned on the
    initial state. To explore its dynamics, below we derive a functional relation for
    ${{\mathcal{P}}^{\mathcal{S}}}(t;s)$ based on the transition probability functions of the system and
    the combinatorial structure of the state space. This relation will serve as a fundamental tool for the
    subsequent analysis of the system's evolution and its long-term behavior. We first establish the
    following result, which provides the explicit form of the GF $\mathcal{M}(s)$ appearing in {\eqref{IIB3.1}}.

    Consider the function
\begin{equation}                \label{IIB4.28}
    \mathcal{M}(t;s):=1-\frac{\Lambda\left(R(t;s)\right)}{\Lambda(q(t))} {\raise1.5pt\hbox{.}}
\end{equation}

\begin{theorem}                      \label{IIBTh7}
    Let $[{f}_{\nu}]$, $[{\mathcal{L}}_{\nu}]$ and $[{\omega}_{\nu}]$ be satisfied. Then
\begin{equation}                \label{IIB4.29}
    \mathcal{M}(s)=\frac{1}{\nu }\left[ \frac{1}{\Lambda (1-s)}-\frac{1}{{{a}_{0}}} \right].
\end{equation}
\end{theorem}

\begin{proof}
    According to condition $[{f}_{\nu}]$, we write
\begin{equation}                \label{IIB4.30}
    \mathcal{M}(t;s)=1-{{\left(\frac{R(t;s)}{q(t)}\right)}^{\nu}}
    \frac{\mathcal{L}\left({1}\big/{R(t;s)}\right)}{\ell\left({1}\big/{q(t)}\right)}{\raise1.5pt\hbox{.}}
\end{equation}
    Relations {\eqref{IIB4.9}} and {\eqref{IIB4.10}} implies that
\begin{equation}                \label{IIB4.31}
    \mathcal{L}\left(\frac{1}{R(t;s)}\right)=
    \mathcal{L}\left(\frac{1}{q(t)}\right)\left(1+\mathcal{O}\left({1}\big/{{{t}^{2}}}\right)\right).
\end{equation}
    Motivated by the seminal work of Slack~{\cite{Slack68}}, we now proceed to develop a continuous-time
    analogue of the classical Slack's function defined as follows:
\begin{equation}                \label{IIB4.32}
    U(t;s)=\nu{t}\frac{q(t)-R(t;s)}{q(t)}{\raise1.5pt\hbox{.}}
\end{equation}
    The limit of $U(t;s)$ as $t\to\infty$ is the function $\mathcal{M}(s)$. Indeed, by transforming
    equation {\eqref{IIB4.32}} into the form
\begin{equation}                \label{IIB4.33}
    R(t;s)=\frac{\mathcal{N}(t)}{{{(\nu{t})}^{{1}/{\nu}}}}\left[1-\frac{U(t;s)}{\nu{t}}\right]
\end{equation}
    we unequivocally obtain the asymptotic behavior of the function $R(t;s)$, as demonstrated in
    equation {\eqref{IIB3.1}}, where $U(t;s)\to\mathcal{M}(s)$ as $t\to\infty$ and $U(t;0)=0$.
    Then it follows from {\eqref{IIB4.28}} and {\eqref{IIB4.30}}--{\eqref{IIB4.33}} that
\[
    \mathcal{M}(t;s)=1-{{\left[1-\frac{U(t;s)}{\nu{t}}\right]}^{\nu}}
    \left(1+\mathcal{O}\left({1}\big/{{{t}^{2}}}\right)\right)=
    \frac{U(t;s)}{t}+\mathcal{O}\left({1}\big/{{{t}^{2}}}\right)
\]
    as $t\to\infty$. Therefore $t\mathcal{M}(t;s)\to\mathcal{M}(s)$ and
\begin{equation}                \label{IIB4.34}
    U(t;s)=t\mathcal{M}(t;s)+\mathcal{O}\left({1}\big/{t}\right) \qquad {\text{as}} \quad {t\to\infty}.
\end{equation}
    Now, combining {\eqref{IIB3.5}}, {\eqref{IIB4.28}} and {\eqref{IIB4.34}}, we obtain
\begin{eqnarray*}
    {\mathcal{M}(s)=\lim_{t\to\infty }t\mathcal{M}(t;s)}
& = & \lim_{t\to\infty}t\left[1-\frac{\Lambda(1-s)}{{{a}_{0}}}
        \frac{{{a}_{0}}\nu{t}+1}{\Lambda(1-s)\nu{t}+1}\right] \nonumber\\
\ \nonumber\\
& = & \frac{1}{\nu\Lambda(1-s)}-\frac{1}{\nu{{a}_{0}}}{\raise1.5pt\hbox{.}}
\end{eqnarray*}

    Thus, the desired relation {\eqref{IIB4.29}} is established.

    The theorem is proved.
\end{proof}

    The following theorem establishes an asymptotic relationship between the probability
    functions $\left\{\textsf{p}_{j}^{\mathcal{S}}(t)\right\}$ and the invariant measure
    $\left\{{{\mu}_{j}}\right\}$ generated by GF $\mathcal{M}(s)$.
\begin{theorem}                      \label{IIBTh8}
    Let conditions $[{f}_{\nu}]$, $[{\mathcal{L}}_{\nu}]$ and $[{\omega}_{\nu}]$ be satisfied. Then
\begin{equation}                  \label{IIB4.35}
    \nu{t}{{\mathcal{P}}^{\mathcal{S}}}(t;s)=\mathcal{M}(s)\cdot\left(1+\rho(t;s)\right),
\end{equation}
    where $\mathcal{M}(s)$ is GF of invariant measure appeared in {\eqref{IIB3.1}},
    is the form of {\eqref{IIB4.29}} and
\[
    \rho(t;s)-\frac{\ln\left[\Lambda(1-s)\nu{t} \right]}{{{\nu}^{3}}t}=o(1)
    \qquad {\text{as}} \quad {t\to\infty}
\]
    uniformly in $s\in[0,1)$.
\end{theorem}

\begin{proof}
     As noted in {\eqref{IIB3.1}}, the GF has the integral form $\mathcal{M}(s)=\int_{0}^{s}{{dx}/{f(x)}}$.
    The backward Kolmogorov equations {\eqref{IIB2.2}} can then be transformed in association with the GF $\mathcal{M}(s)$ as follows:
\begin{equation}                  \label{IIB4.36}
    \mathcal{M}\left(F(t;s)\right)=\mathcal{M}(s)+t.
\end{equation}
    We now combine formulas {\eqref{IIB4.29}} and {\eqref{IIB4.36}} and obtain
\begin{align*}
\left\{
\begin{array}{l}
    \displaystyle \frac{1}{\Lambda (R(t;s))}-\frac{1}{\Lambda (1-s)}=\nu{t},
\vspace{2.8mm}  \\
    \displaystyle \frac{1}{\Lambda (q(t))}-\frac{1}{{{a}_{0}}}=\nu{t}.  \\
\end{array}
\right.
\end{align*}
    Term-by-term subtraction of these equalities gives
\[
    \frac{1}{\Lambda(R(t;s))}-\frac{1}{\Lambda(q(t))}=\nu\mathcal{M}(s)
\]
    and hence
\begin{equation}                  \label{IIB4.37}
    \mathcal{M}(t;s)=1-\frac{\Lambda\left(R(t;s)\right)}{\Lambda(q(t))}=\nu\Lambda(R(t;s))\mathcal{M}(s).
\end{equation}
    Furthermore, based on our assumptions and notation, we write
\begin{equation}                  \label{IIB4.38}
{{\mathcal{P}}^{\mathcal{S}}}(t;s)={\frac{1}{\nu{t}}}U(t;s).
\end{equation}
    Relations {\eqref{IIB4.34}}, {\eqref{IIB4.37}}, and {\eqref{IIB4.38}}, when considered together,
    yield the following relation:
\begin{equation}                  \label{IIB4.39}
    {{\mathcal{P}}^{\mathcal{S}}}(t;s)=\mathcal{M}(s)\Lambda(R(t;s))+\mathcal{O}\left({1}\big/{{{t}^{2}}}\right)
    \qquad {\text{as}} \quad {t\to\infty}.
\end{equation}
    Now relation {\eqref{IIB4.35}} follows from {\eqref{IIB4.39}} and {\eqref{IIB4.6}}.

    The theorem is proved.
\end{proof}

\begin{remark}                  \label{IIBRem5}
      The system $\left\{{{Z}^{\mathcal{S}}}(t)\right\}$ is an ergodic Markov chain; see {\cite{ImIJSA14}}. According
      to equation {\eqref{IIB4.36}}, coefficients $\left\{{{\mu}_{j}}\right\}$ generated by the power series
      expansion $\mathcal{M}(s)=\sum_{j\in\mathcal{S}}{{{\mu}_{j}}{{s}^{j}}}$ is the invariant measures for
      this chain. Furthermore, by the Hardy-Littlewood Tauberian theorem, relation {\eqref{IIB4.29}} yields
\[
    \sum_{j=1}^{n}{{{\mu}_{j}}}=\frac{1}{{{\nu}^{2}}\Gamma(\nu)}{{n}^{\nu}}{{\mathcal{L}}_{\mu}}(n),
\]
    where $\Gamma(\cdot)$ is the Euler's Gamma function and ${{{\mathcal{L}}_{\mu}}(t)}/{\mathcal{L}(t)}\to{1}$ as $t\to\infty$.
\end{remark}

    We now proceed to formulate a novel analogue of the monotone convergence theorem for MBS, accompanied
    by an explicit estimate for the convergence rate. Let
\[
    \mathcal{V}(t;s):=\sum_{j\in\mathcal{S}}{\frac{{{\textsf{p}}_j}(t)}{{{\textsf{p}}_1}(t)}{s^j}}.
\]

\begin{theorem}                      \label{IIBTh9}
    Let conditions $[{f}_{\nu}]$, $[{\mathcal{L}}_{\nu}]$ and $[{\omega}_{\nu}]$ be satisfied. Then the
    functions $\mathcal{V}(t;s)$ converges to $\mathcal{V}(s)={a_0}\mathcal{M}(s)$ as $t\to\infty$. Denoting
    the power series expansion of $\mathcal{V}(s)$ by $\mathcal{V}(s)=\sum_{j\in\mathcal{S}}{{{v}_{j}}{{s}^{j}}}$,
    the ratios ${{\textsf{p}_j}(t)}\big/{{\textsf{p}_1}(t)}$ are given by
\begin{equation}                  \label{IIB4.40}
    \frac{{{\textsf{p}}_j}(t)}{{{\textsf{p}}_{1}}(t)}={{v}_{j}}\cdot\left(1+\frac{\ln[\Lambda(1-s)\nu{t}]}
    {{{\nu}^{3}}t}\left(1+o(1)\right)\right) \qquad {\text{as}} \quad {t\to\infty}.
\end{equation}
\end{theorem}

\begin{proof}
    Write
\[
    \mathcal{V}(t;s)=\sum_{j\in\mathcal{S}}{\frac{{{p}_{j}}(t)}{{{p}_{1}}(t)}{{s}^{j}}}=
    \frac{q(t)}{{{p}_{1}}(t)}\left[1-\frac{R(t;s)}{q(t)}\right]=\frac{q(t)}{{{p}_{1}}(t)}{{\mathcal{P}}^{\mathcal{S}}}(t;s).
\]
    Combining this relation with {\eqref{IIB3.22}} and {\eqref{IIB4.35}} yields
\[
    \mathcal{V}(t;s)={{a}_{0}}\mathcal{M}(s)\cdot\left(1+\rho(t;s)\right)
    \left(1-\frac{\ln{{a}_{0}}\nu{t}}{{{\nu}^{2}}t}\left(1+o(1)\right)\right).
\]
    Then {\eqref{IIB4.40}} directly follows.
\end{proof}

\section{Conclusion}             \label{IIBSec5}

    This paper develops refined asymptotic results for continuous-time Markov branching systems in the critical
    regime. In contrast to classical works, which rely on finiteness of higher-order moments, our analysis is
    conducted under significantly weaker assumptions, formulated through the Basic Axioms involving slowly
    varying functions with remainder terms in the sense of Karamata.

    The main results include enhanced asymptotic expansions for the function $R(t;s)$, with explicit logarithmic
    corrections ({\hyperref[IIBLem1]{Lemma}~\ref{IIBLem1}} and {\hyperref[IIBTh1]{Theorem}~\ref{IIBTh1}}), and
    a detailed description of the asymptotic behavior of the survival probability $q(t)$.
    {\hyperref[IIBTh2]{Theorem}~\ref{IIBTh2}} and {\hyperref[IIBTh3]{Theorem}~\ref{IIBTh3}} establish the
    asymptotic relation between the survival probability and local transition probabilities of MBS without
    immigration, revealing that $\textsf{p}_1(t)$ belongs to the class of slowly varying functions with
    remainder -- a property not captured by classical approaches.

    For the transient MBIS, we construct the limiting GF $\mathcal{U}(s)$ in {\hyperref[IIBTh4]{Theorem}~\ref{IIBTh4}}
    and prove convergence of the normalized GF $e^{\mathcal{T}(t)}\mathcal{P}(t;s)$ to $\mathcal{U}(s)$ with
    an explicit rate. {\hyperref[IIBTh5]{Theorem}~\ref{IIBTh5}} and {\hyperref[IIBTh6]{Theorem}~\ref{IIBTh6}}
    provide a complete characterization of the invariant measure via ratio limits and an asymptotic
    expansion of the scaling factor $e^{\mathcal{T}(t)} p_{00}(t)$.

    Theorems~\ref{IIBTh7}--\ref{IIBTh9} develop the asymptotic behavior of the limiting
    GF $\mathcal{M}(s)$ associated with the positively conditioned MBS. They provide a continuous-time
    analogue of Slack's classical function from discrete critical branching theory. In particular, they characterize
    the invariant distribution of the system under positive survival conditioning and provide precise asymptotic
    expansions that quantify the effect of heavy tails and immigration. This contributes a novel perspective to
    the study of conditional limit theorems in continuous-time settings with infinite variance components.
    Unlike classical conditioning approaches that rely on finite-variance settings or discrete-time formulations,
    our analysis leverages the structure of the continuous-time model and the refined asymptotic tools involving
    slowly varying functions with remainders. This technical apparatus proves essential in deriving sharp expansions
    for $\mathcal{M}(s)$ and in capturing second-order effects. As a result, the limiting function not only
    generalizes earlier discrete-time results, but also reveals subtle regularity features that emerge uniquely
    in the infinite-variance, continuous framework.

    Collectively, these results constitute a significant extension of existing asymptotic theory for
    branching systems, particularly in settings where classical moment conditions fail. The methodology
    developed herein not only generalizes and sharpens existing limit theorems but also provides a framework
    for analyzing systems with nonstandard scaling and heavy-tailed characteristics.

\section{Appendixes}             \label{IIBSec6}

\subsection{Empirical Illustration of Survival Probability in Critical MBS.}

    The {\hyperref[FIqsurvival]{Figure}~\ref{FIqsurvival}} and {\hyperref[FIp1local]{Figure}~\ref{FIp1local}} provide
    an empirical illustration of the asymptotic behaviors described in {\hyperref[IIBTh1]{Theorem}~\ref{IIBTh1}} and
    {\hyperref[IIBTh2]{Theorem}~\ref{IIBTh2}} the structural assumption $[f_{\nu}]$. This axiom ensures that the
    offspring distribution lies in the domain of attraction of a $(1+\nu)$-stable law, so that the corresponding
    MBS exhibits a heavy-tailed reproduction mechanism and polynomial decay of survival-related probabilities.
    In the graphical analysis, two representative forms of $\mathcal{L}(x)$ are considered which, when expressed
    in terms of $t$, correspond to the slowly varying factors
\[
    \mathcal{N}(t)=1+\frac{1}{2\log(t+1)} \qquad {\text{and}} \qquad  \mathcal{N}(t)=1+\frac{\log(t+1)}{t^{\nu}}{\raise1.5pt\hbox{.}}
\]
    These choices allow us to explore the influence of $\mathcal{L}(\cdot)$ on the pre-asymptotic dynamics
    while preserving the principal scaling structure dictated by axiom $[f_{\nu}]$. The survival probability $q(t)$ and
    the local probability $p_1(t)$, plotted on a logarithmic scale, follow the expected power-law decay with
    exponents $-1/\nu$ and $-1/\nu - 1$, respectively, as $t \to \infty$. The deviations from a straight line
    in the finite-time range are entirely determined by the specific form of $\mathcal{N}(t)$,
    illustrating how slowly varying components modulate the main asymptotic law.

    This example substantiates the asymptotic behavior derived in {\hyperref[IIBSec3]{Section}~\ref{IIBSec3}} and
    illustrates how the survival probability reflects the structural influence of heavy-tailed branching dynamics.
    It also provides a practical reference for researchers applying such models in fields such as population
    biology and epidemiology. Similar illustrations arise from the corresponding asymptotic properties of MBIS
    established in {\hyperref[IIBTh4]{Theorem}~\ref{IIBTh4}} and {\hyperref[IIBTh5]{Theorem}~\ref{IIBTh5}}.
\begin{figure}[h!]
\centering
\foreach \n/\a in {0.2/0.9, 0.9/0.2}{
  \begin{tikzpicture}
    \begin{axis}[
      width=12cm,
      height=7cm,
      xlabel={$t$},
      ylabel={$q(t)$},
      title={Survival probability $q(t)$ for $\nu=\n,\ a_0=\a$},
      legend style={at={(0.5,-0.15)}, anchor=north, legend columns=2},
      grid=major,
      ymode=log
    ]
    \addplot[green!60!black, thick, domain=5:100, samples=700]
      {(1 + 0.5/ln(x+1))/((\n*x)^(1/\n))*(1 + ln(\a*\n*x)/(\n^3*x))};
    \addlegendentry{$\mathcal{N}(t)=1+\frac{1}{2\log(t+1)}$}

    \addplot[magenta, very thick, dashed, domain=5:100, samples=700]
      {(1 + ln(x+1)/(x^\n))/((\n*x)^(1/\n))*(1 + ln(\a*\n*x)/(\n^3*x))};
    \addlegendentry{$\mathcal{N}(t)=1+\frac{\log(t+1)}{t^{\nu}}$}
    \end{axis}
  \end{tikzpicture}

  \vspace{8mm} 
}
\caption{Illustrations of the asymptotic behaviour of the survival probability $q(t)$ for different slowly varying factors $\mathcal{N}(t)$ and parameter sets $(\nu,a_0)$.}
\label{FIqsurvival}
\end{figure}
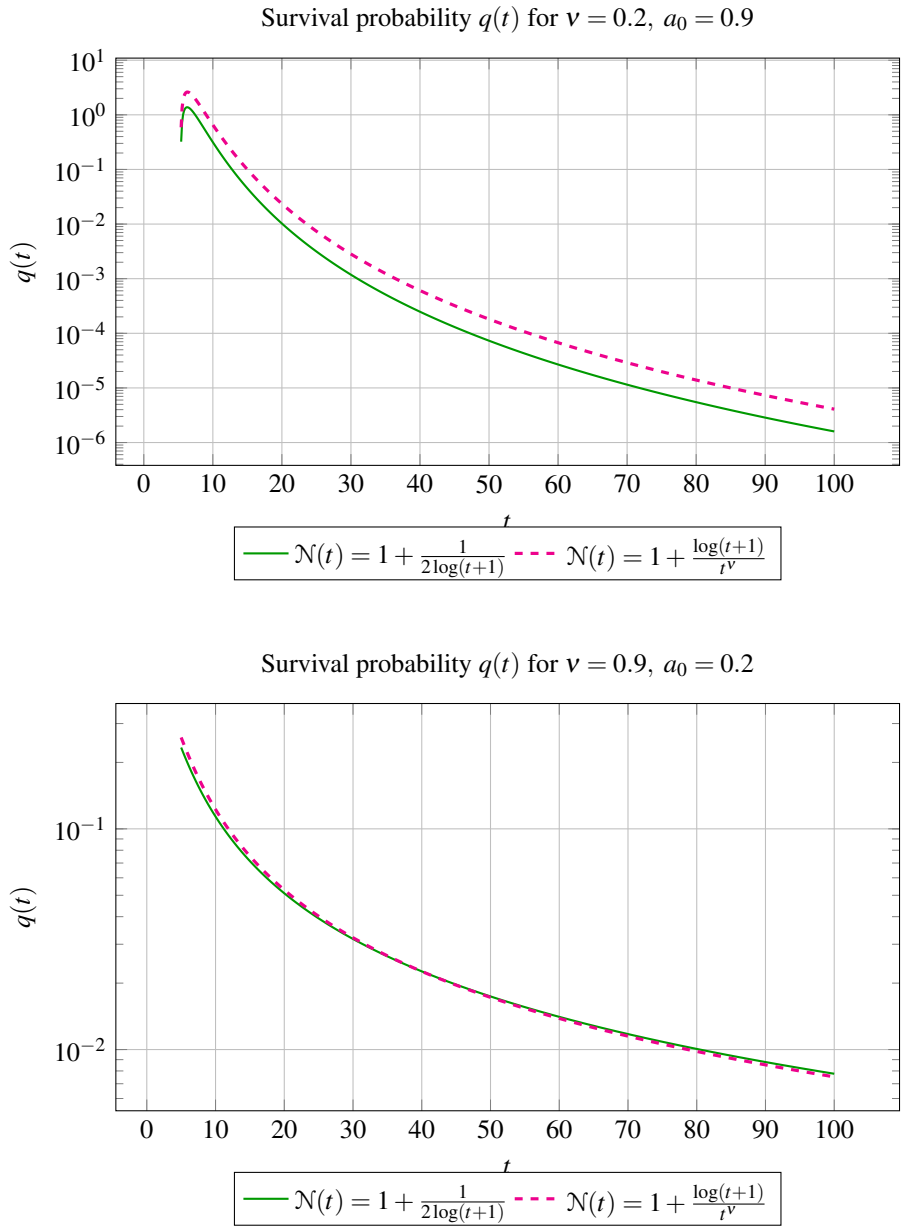

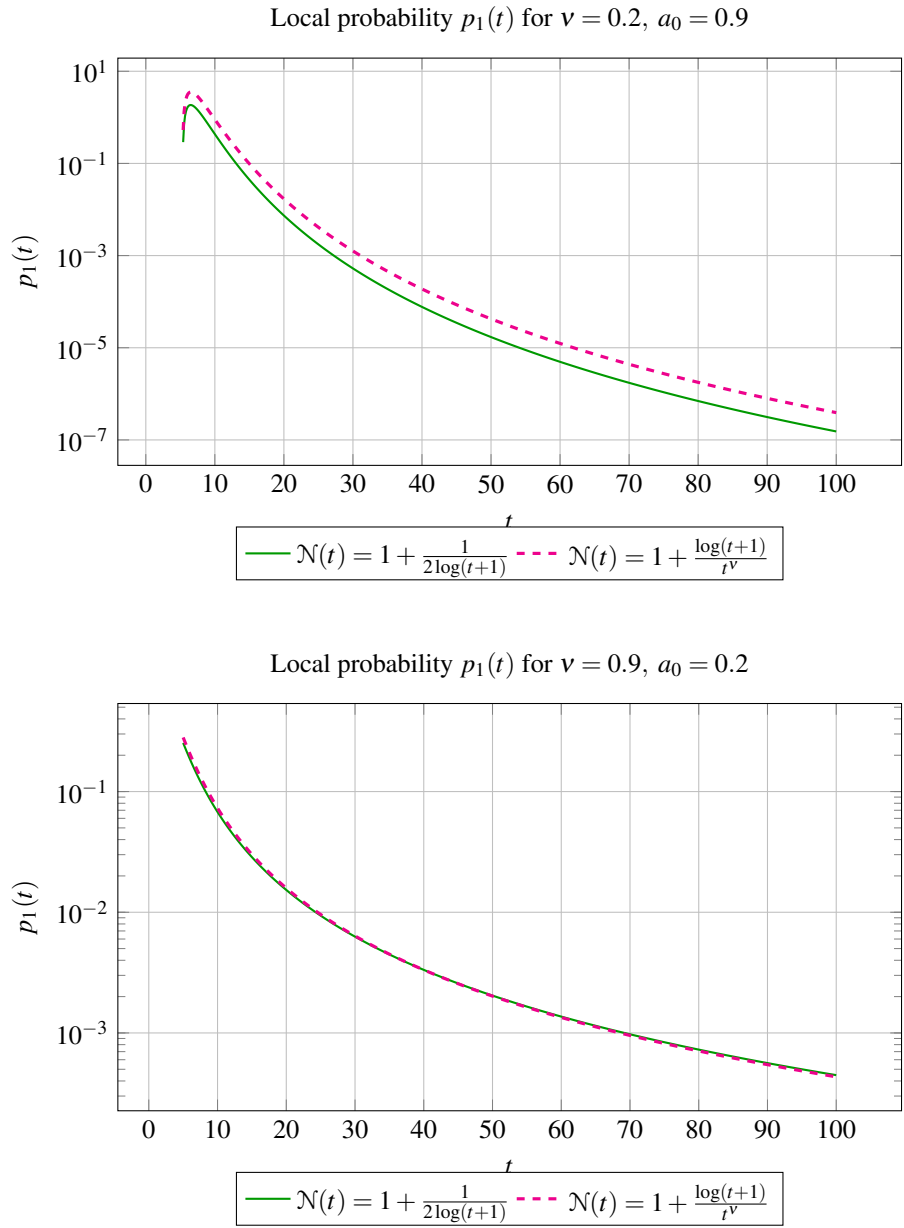
\begin{figure}[h!]
\centering
\foreach \n/\a in {0.2/0.9, 0.9/0.2}{
  \begin{tikzpicture}
    \begin{axis}[
      width=12cm,
      height=7cm,
      xlabel={$t$},
      ylabel={$p_1(t)$},
      title={Local probability $p_1(t)$ for $\nu=\n,\ a_0=\a$},
      legend style={at={(0.5,-0.15)}, anchor=north, legend columns=2},
      grid=major,
      ymode=log
    ]
    \addplot[green!60!black, thick, domain=5:100, samples=700]
      {(1 + 0.5/ln(x+1))/((\n*x)^(1/\n))*(1 + ln(\a*\n*x)/(\n^3*x))
       * (1/(\a*\n*x)*(1 + ln(\a*\n*x)/(\n^2*x)))};
    \addlegendentry{$\mathcal{N}(t)=1+\frac{1}{2\log(t+1)}$}

    \addplot[magenta, very thick, dashed, domain=5:100, samples=700]
      {(1 + ln(x+1)/(x^\n))/((\n*x)^(1/\n))*(1 + ln(\a*\n*x)/(\n^3*x))
       * (1/(\a*\n*x)*(1 + ln(\a*\n*x)/(\n^2*x)))};
    \addlegendentry{$\mathcal{N}(t)=1+\frac{\log(t+1)}{t^{\nu}}$}
    \end{axis}
  \end{tikzpicture}

  \vspace{8mm} 
}
\caption{Illustrations of the asymptotic behaviour of the local probability $p_1(t)$ for different slowly varying factors $\mathcal{N}(t)$ and parameter sets $(\nu,a_0)$.}
\label{FIp1local}
\end{figure}

\subsection{On Extensions to Non-Homogeneous Systems.}
    The theoretical framework developed in this paper assumes time-homogeneous branching and immigration rates,
    which allows the use of classical GFs and the theory of slowly varying functions in the sense of Karamata.
    Nevertheless, many practical systems—ranging from seasonal population dynamics to adaptive epidemiological
    models—exhibit time-varying or state-dependent behavior. One natural extension is to consider branching and
    immigration GFs of the form
\[
    f_t(s)=\sum_{k\in{\mathbb{N}_0}}{a_k(t)s^k} \qquad \text{and}  \qquad   h_t(s)=\sum_{j\in{\mathbb{N}_0}}{b_j(t)s^j},
\]
    where the coefficients $a_k(t)$ and $b_j(t)$ vary with time, possibly in a slowly varying or periodic manner.
    Alternatively, one may introduce state-dependence by letting $a_k=a_k(x)$ depend on the current population
    size $x$, leading to a state-inhomogeneous generator. In these cases, the Kolmogorov backward and forward
    equations must be reformulated to incorporate explicit time- or state-dependence in the infinitesimal
    generator. The analysis of such systems would likely require the development of non-stationary analogues
    of slowly varying functions and the adaptation of asymptotic tools to the time-inhomogeneous setting.
    Although such generalizations lie beyond the scope of the present work, they represent a natural and
    meaningful direction for future research, particularly in modeling systems with seasonal migration,
    aging structures, or feedback-controlled reproduction.

\subsection{Enhanced Accessibility.}

\paragraph{Summary of Key Asymptotic Expansions}
    The table below lists key quantities, including survival probabilities, GFs, and invariant measures,
    along with their corresponding asymptotic forms. These results are expressed in terms of the model
    parameters \(\nu\), \(\delta\), and the remainder terms associated with slowly varying functions.

\begin{table}[h!]
\centering
\renewcommand{\arraystretch}{2.4}
\begin{tabular}{|c|p{10cm}|}
\hline
\textbf{Quantity} & \textbf{Asymptotic Expression} \\
\hline
$R(t;s)=1-F(t;s)$ &
$R(t;s) \sim \dfrac{\mathcal{N}(t)}{(\nu{t})^{1/\nu}} \left(1 + \dfrac{\ln[\Lambda(1-s)\nu{t}]}{\nu^3{t}}\right)
    \quad \text{as} \quad {t\to\infty}$ \\
\hline
$q(t)=R(t;0)$ &
$q(t) \sim \dfrac{\mathcal{N}(t)}{(\nu{t})^{1/\nu}} \left(1 + \dfrac{\ln(a_0 \nu{t})}{\nu^3{t}}\right)
    \quad \text{as} \quad {t\to\infty}$ \\
\hline
$\textsf{p}_1(t)$ &
$\textsf{p}_1(t) \sim \dfrac{q(t)}{a_0 \nu{t}} \left(1 + \dfrac{\ln(a_0 \nu{t})}{\nu^2{t}}\right)
    \quad \text{as} \quad {t\to\infty}$ \\
\hline
Invariant GF $\mathcal{U}(s)$ &
$\ln\mathcal{U}(s)=\dfrac{1}{(1-s)^{|\gamma|}}+\int_{1/(1-s)}^{\infty} \dfrac{|\gamma|-\textsf{L}(u)}{u^{1-|\gamma|}}du$ \\
\hline
Invariant GF $\pi(s)$ &
$\ln\pi(s)=\dfrac{1}{(1-s)^{|\gamma|}}\textsf{L}_v\left(\dfrac{1}{1-s}\right)$ \\
\hline
Invariant GF $\mathcal{M}(s)$ &
$\mathcal{M}(s)=\dfrac{1}{\nu}\left(\dfrac{1}{\Lambda(1-s)}-\dfrac{1}{a_0}\right)$ \\
\hline
\end{tabular}
\end{table}

\paragraph{Interpretation of the Parameter \(\gamma=\delta-\nu\)}
    The sign of the parameter \(\gamma\) plays a decisive role in determining the global behavior of the system. In particular:

\begin{itemize}
\item if \(\gamma < 0\): the system is \emph{transient}, meaning the population tends to escape to infinity and does not stabilize;
\item if \(\gamma > 0\): the system is \emph{positive recurrent}, meaning that a stationary behavior emerges and invariant measures exist;
\item if \(\gamma = 0\): the system exhibits \emph{critical behavior}, and often requires alternative scaling (e.g., leads to a Markov $Q$-process).
\end{itemize}

    From a modeling perspective, \(\nu\) controls the tail heaviness of the reproduction law, while \(\delta\) reflects
    the immigration input. Their difference \(\gamma\) measures the imbalance between these two mechanisms.

\paragraph{Special Case: \(\nu=1\)}
    The choice \(\nu = 1\) corresponds to offspring distributions with finite second moments. In this classical setting,
    the function \(f(s)\) behaves near \(s = 1\) as \(f(s) \sim 1 - s\), and the asymptotics reduce to polynomial or
    exponential forms. For example, the survival probability simplifies to:
\[
    q(t) \sim \dfrac{C}{t} \qquad \text{as} \quad {t\to\infty},
\]
    where \(C\) depends on the immigration parameters. This scenario aligns with classical results such as the
    Yaglom or Sevastyanov theorems, providing a bridge between the heavy-tailed framework developed in this
    paper and the well-known finite-variance theory.

\subsection{Remark on Two-Scale Fractal Perspectives.}

    The presence of slowly varying functions in our asymptotic analysis hints at deeper structural features of
    the model, where different mechanisms operate across distinct temporal or population scales. This observation
    resonates with the spirit of two-scale fractal theory, which describes systems whose dynamics shift character
    between scales — smooth at the macro level, yet potentially irregular or self-similar in finer detail. Although
    our results focus on global asymptotics and limit shapes, one might conjecture that fluctuations around these
    limits, or the fine structure of invariant measures, could exhibit patterns more naturally understood through
    multi-scale or fractal lenses. While such exploration lies beyond the scope of the present study, it may lead
    to new qualitative insights into how randomness, scaling, and structure interact within complex stochastic systems.
    This perspective finds a parallel in {\cite{He2Liu24}} and {\cite{HeLiu24}}, where fractal modeling approaches
    are employed to analyze porous materials, offering complementary viewpoints that resonate with the branching
    systems considered here.

\subsection{General Perspectives.}

    The results of the present work open several promising directions for further research. Below, we highlight
    a few perspectives that may enrich or extend the framework developed herein.

\begin{itemize}
\item \textbf{Two-Scale Fractal Structures.} As discussed above, the involvement of slowly varying functions
    suggests potential connections with two-scale fractal theory. Investigating whether the transition dynamics
    or invariant distributions exhibit multiscale irregularities could lead to new insights at the intersection
    of probability and fractal analysis.

\item \textbf{Non-Homogeneous Dynamics.} Extending the model to accommodate time-dependent or state-dependent
    branching and immigration rates would reflect more realistic settings and pose challenging mathematical
    questions related to time-inhomogeneous renewal theory.

\item \textbf{Functional Limit Theorems.} While we focused on pointwise asymptotics, it would be of interest to derive
    functional limit theorems (e.g., weak convergence in Skorokhod spaces), especially in critical or near-critical regimes.

\item \textbf{Empirical Applications.} Further work might explore how the proposed framework can be calibrated
    to empirical data in demography, epidemiology, or migration, and whether the asymptotic invariants can
    be meaningfully estimated from observed trajectories.
\end{itemize}


\section*{Conflict of Interest}

    The authors declare that they have no conflicts of interest.

\section*{Data Availability}

    All data generated or analyzed during this study are included in this published article.

\section*{Acknowledgment}

    This research work was carried out within the framework of the project program FL-8824063218 of
    the Ministry of Higher Education, Science, and Innovations of the Republic of Uzbekistan.

\end{document}